\documentclass{amsart}

\usepackage[utf8]{inputenc}

\usepackage{tikz-cd}
\usepackage{quiver}

\usepackage{epigraph}

\usepackage{hyperref}
\hypersetup{
pdftitle={Cartier duality via Mittag-Leffler modules},
pdfauthor={Dima Arinkin, Joshua Mundinger},
pdfkeywords={Cartier duality, Mittag-Leffler module, ind-scheme, Fourier-Mukai transform}
}

\usepackage{enumerate}
\usepackage{xcolor}

\usepackage[
  style=alphabetic,
  citestyle=alphabetic,
  maxnames=100, 
  doi=true,
  isbn = false,
  giveninits = true, 
  datamodel = mrnumber 
]{biblatex}
\DeclareFieldFormat{mrnumber}{%
  MR\addcolon\space
    {
      \href{http://www.ams.org/mathscinet-getitem?mr=#1}
    {\nolinkurl{#1}}
    }
}
\usepackage{xpatch}
\xapptobibmacro{finentry}{\setunit{\par} \printfield{mrnumber}}{}{}

\usepackage{amsthm}

\theoremstyle{plain}
\newtheorem{theorem}{Theorem}[subsection]
\newtheorem{lemma}[theorem]{Lemma}
\newtheorem{proposition}[theorem]{Proposition}
\newtheorem{corollary}[theorem]{Corollary}

\newtheorem{theoremalpha}{Theorem}

\theoremstyle{definition}
\newtheorem{definition}[theorem]{Definition}
\newtheorem{example}[theorem]{Example}

\newtheorem{remark}[theorem]{Remark}
\newtheorem{warning}[theorem]{Warning}

\usepackage{amsmath,amssymb,mathrsfs}

\newcommand{\cA}{\mathcal{A}}
\newcommand{\cX}{\mathcal{X}}
\newcommand{\cY}{\mathcal{Y}}
\newcommand{\C}{\mathbb{C}}
\newcommand{\CC}{\mathcal{C}}

\newcommand{\Einfty}{{\mathbb{E}_\infty}}

\newcommand{\OO}{\mathcal O}
\newcommand{\ZZ}{\mathbb{Z}}

\DeclareMathOperator{\Alg}{Alg}

\DeclareMathOperator{\CAlg}{CAlg}

\DeclareMathOperator{\Coh}{Coh}

\newcommand{\colim}{\varinjlim}

\DeclareMathOperator{\End}{End}
\DeclareMathOperator{\Ext}{Ext}
\DeclareMathOperator{\Extu}{\underline{Ext}}

\DeclareMathOperator{\Fun}{Fun}

\DeclareMathOperator{\Hom}{Hom}
\DeclareMathOperator{\Homs}{\underline{Hom}}

\DeclareMathOperator{\Op}{Op}
\DeclareMathOperator{\Pic}{Pic}
\DeclareMathOperator{\PreStk}{PreStk}
\DeclareMathOperator{\Pro}{Pro}
\DeclareMathOperator{\ProMod}{\Pro\Mod}

\DeclareMathOperator{\Mod}{Mod}

\DeclareMathOperator{\IndCoh}{IndCoh}
\DeclareMathOperator{\IndSch}{IndSch}
\DeclareMathOperator{\Perf}{Perf}

\DeclareMathOperator{\Shv}{Shv_{\mathit{R}}}
\DeclareMathOperator{\Shvab}{Shv_{\mathit{R}}^{ab}}
\DeclareMathOperator{\Spec}{Spec}
\DeclareMathOperator{\SpInd}{SpInd}
\DeclareMathOperator{\Spc}{Spc}
\DeclareMathOperator{\Sym}{Sym}
\DeclareMathOperator{\Tor}{Tor}
\DeclareMathOperator{\Tot}{Tot}
\DeclareMathOperator{\QCoh}{QCoh}

\newcommand{\DCAlg}{\mathrm{DCAlg}}

\renewcommand{\lim}{\varprojlim}

\newcommand\Sets{\mathcal{S}\mathit{ets}}

\renewcommand{\t}[1]{{}^{\leq #1}\!}

\newcommand\Gm{\mathbb{G}_m}
\newcommand{\Ga}{\mathbb{G}_a}

\newcommand{\series}[1]{[\![#1]\!]}
\newcommand{\lseries}[1]{(\!(#1)\!)}

\newcommand{\tensor}[1]{\underset{#1}{\otimes}}
\newcommand{\Lotimes}[1]{\underset{#1}{\overset{L}{\otimes}}}

\newcommand{\EXT}{\mathtt{EXT}}

\newcommand{\unit}{\mathrm{unit}}

\newcommand{\ML}{M\!L}

\newcommand{\stCats}{\mathrm{Cat_\infty}^{\mathrm{st,cocompl}}_{\mathrm{cont}}}

\title{Cartier duality via Mittag-Leffler modules}
\author{Dima Arinkin, Joshua Mundinger}
\date{December 15, 2025}

\begin{document}

\begin{abstract}
    We construct the Cartier duality equivalence for affine commutative group schemes $G$ whose coordinate ring is a flat Mittag-Leffler module over an arbitrary base ring $R$. The dual $G^\vee$ of $G$ turns out to be an ind-finite ind-scheme over $R$. When $R$ is Noetherian and admits a dualizing complex, we construct a Fourier-Mukai transform between quasicoherent derived categories of $G$ and of $BG^\vee$ and also between those of $G^\vee$ and $BG$.
\end{abstract}

\maketitle


\section{Introduction}

\subsection{Extended abstract}

If $G$ is a finite abelian group, then the set of homomorphisms $G \to \C^\times$ is again a finite abelian group $G^\vee$. It is a fundamental observation of Pontryagin that the assignment $G \rightsquigarrow G^\vee$ is an antiequivalence of categories. 
In algebraic geometry, motivated by the study of algebraic groups in positive characteristic, Cartier introduced the dual $G^\vee$ of a finite group scheme $G$ over a field, given by the formula $G^\vee = \Homs(G, \Gm)$ \cite{Car62}. Again the assignment $G \rightsquigarrow G^\vee$ is an antiequivalence of categories. 
This extends verbatim to finite locally free families of group schemes \cite[VIIA]{SGA3I}.

In both  settings, the duality is in fact defined on a larger category of abelian groups. 
In the topological setting, Pontryagin duality is an antiequivalence on locally compact abelian groups.
In the algebraic setting, the symmetry is broken: duality becomes an antiequivalence between categories of affine commutative group schemes over a field $k$ and commutative ``formal groups'' over $k$, which in modern terms are group ind-finite ind-schemes over $k$ \cite{Car62, Die73}. 

One of the main goals of this paper is to explain an extension of the Cartier duality equivalence to group schemes over an arbitrary base. 
Cartier observed that a finite algebraic group $G$ is exactly the spectrum of a finite-dimensional commutative cocommutative Hopf algebra $A$, and the linear dual $A^\vee$ of a commutative cocommutative Hopf algebra is again a commutative cocommutative Hopf algebra.  The dual group $G^\vee$ is then the spectrum of $A^\vee$. Linear duality is an antiequivalence of the category of finite-dimensional vector spaces and induces the Cartier duality.
Thus, the crux of the matter is to give an appropriate duality theory for modules over a ring. For us, the answer is the theory of \emph{flat Mittag-Leffler modules}, introduced by Raynaud and Gruson in the seminal paper \cite{RG71}. Informally, Drinfeld calls flat Mittag-Leffler modules ``projective modules with a human face''. 
In particular, projective modules are flat Mittag-Leffler, and countably generated flat Mittag-Leffler modules are projective. 

In this paper, we construct the Cartier duality equivalence for affine commutative group schemes $G$ whose coordinate ring is a flat Mittag-Leffler module over the base. The dual of $G$ turns out to be an ind-finite ind-scheme over the same base, and the duality results in an antiequivalence of categories.

We then go one categorical level higher, which leads to two types of duality: geometric 1-duality of group stacks and categorical 1-duality (the Fourier-Mukai transform). We show in what sense both types of duality hold between $G$ and the classifying stack $BG^\vee$ when $G$ is flat Mittag-Leffler. 
More generally, we prove an equivalence of categories of sheaves on $G/H^\vee$ and $H/G^\vee$ when $G$ and $H$ are flat Mittag-Leffler; this extends Laumon's Fourier transform for generalized 1-motives.
We also consider the duality between $G^\vee$ and $BG$, which turns out to be more complicated; we prove partial results in this direction. 

Our primary motivation is in applications to geometric class field theory, such as appropriate versions of the Contou-Carrère pairing in families.\footnote{This is similar to Campbell and Hayash's study of geometric class field theory \cite{CH21} We focus on categorical 1-duality, but even in the geometric setting, we provide new analysis of the geometric 1-dual of $G$ using torsors in $h$-topology.} This will be considered in future work. The present work does not deal with geometric class field theory, and we expect the results will have wider applications.

\subsection{Summary of main results}

Let $R$ be a commutative ring.
We work over the base $\Spec R$.
The results of this paper belong to two different categorical levels. At the 0-level, the main result is the following theorem:

\begin{theoremalpha}\label{maintheorem: cartier}
    (Corollary \ref{corollary: cartier duality is an antiequivalence})
    The assignment $G \mapsto G^\vee = \Homs(G, \Gm)$ 
    is an antiequivalence of categories between flat Mittag-Leffler affine commutative group schemes over $R$ and coflat commutative group ind-finite ind-schemes over $R$.
\end{theoremalpha}

See Definition \ref{definition: coflat} below for the meaning of the adjective ``coflat".
The main input to Theorem \ref{maintheorem: cartier} is the linear duality for flat Mittag-Leffler modules, following Drinfeld. This is developed in §\ref{section: linear duality}.

It is instructive to compare our results with \cite{ASS07}. Similar to our setting, the Cartier duality is deduced from a linear duality statement. However, their linear duality is realized in terms of presheaves of $\OO$-modules on the category of affine schemes over the base (which they call $R$-module functors). For us, linear duality is the duality between flat modules and pro-finite pro-projective modules (see Corollary \ref{cor: linear duality equivalence}). As a minor point, they work with projective coordinate rings while we consider the more general setting of flat Mittag-Leffler coordinate rings.

At the 1-categorical level, 
the first-named author posited in \cite[Appendix A]{DP08} that there are two natural types of duality between commutative group stacks $\mathcal X$ and $\mathcal Y$.
Given a group stack $\mathcal X$, let its 1-dual be $\mathcal X^D = \Homs(\mathcal X, B\Gm)$.
Given a pairing $P: \mathcal X \times \mathcal Y \to B\Gm$, we say that $P$ induces \emph{geometric 1-duality} if it is perfect in the sense that the induced maps  $\mathcal Y \to \mathcal{X}^D$ and $\mathcal{X} \to \mathcal{Y}^D$ are equivalences.
We say $P$ induces \emph{categorical 1-duality} if the Fourier-Mukai transform with kernel $P$ is an equivalence $\QCoh(\mathcal X) \simeq \QCoh(\mathcal{Y})$ of derived categories of quasi-coherent sheaves.\footnote{In \cite[Appendix A]{DP08}, categorical 1-duality is formulated using bounded derived categories.}

\begin{example}
    If $A$ is an abelian variety, then the Barsotti-Weil formula states that $A^D$ is the dual abelian variety to $A$.
    Since $A^D$ is again an abelian variety, 
    the Poincaré pairing $P: A \times A^D \to B\Gm$ is perfect.
    Mukai showed that the Fourier-Mukai transform
    \[ \QCoh(A) \to \QCoh(A^D)\]
     with kernel $P$ is an equivalence \cite{Muk81}. Thus the Poincaré pairing induces both a geometric and categorical 1-duality.
\end{example}




Let $G$ be a flat Mittag-Leffler affine commutative group scheme over $R$. 
Cartier duality gives rise to a pairing 
\[G \times BG^\vee \to B\Gm.\]
Surprisingly, this pairing is not perfect in general. For example, when $G = \Ga$ over a field of characteristic zero, the induced map $B\Ga^\vee \to \Ga^D$ is not an equivalence because $\Extu^1(\Ga,\Gm) \neq 0$.
Our point of view is that the flat topology used to define $B\Ga^\vee$ is not suitable for ind-schemes such as $\Ga^\vee$.
In §\ref{section: geometric 1-dual} we define a version of the classifying stack using $h$-topology and show (under certain technical assumptions) that for flat Mittag-Leffler $G$, the 1-dual $G^D$ is the $h$-classifying stack of $G^\vee$.

More generally, given $G$ and $H$ and a pairing $G^\vee \times H^\vee \to \Gm$, we obtain a pairing of quotient stacks
\[[G/H^\vee] \times [H/G^\vee] \to B\Gm.\]
If $H$ is trivial, then $[H/G^\vee]$ is the classifying stack of $G^\vee$. 
Even though this pairing may not induce geometric 1-duality, the derived categories of sheaves are still equivalent.
\begin{theoremalpha}\label{maintheorem: fourier for 1-motives}
    (Theorem \ref{theorem: FM for 1-motives})
    Let $R$ be a Noetherian ring admitting a dualizing complex.    
    Let $G$ and $H$ be flat Mittag-Leffler affine commutative group schemes of finite type over $R$, equipped with a bilinear pairing $G^\vee \times H^\vee \to \Gm$.
    Then there is an equivalence of $\infty$-categories
    \[ 
        \QCoh([G/H^\vee])  \simeq \QCoh([H/G^\vee]).
    \]
\end{theoremalpha}

Theorem \ref{maintheorem: fourier for 1-motives} may be considered as a version\footnote{Laumon's generalized 1-motives include abelian varieties as subfactors, but we only allow for affine group schemes and their duals.} of Laumon's Fourier-Mukai transform in families \cite{Lau96}.
Theorem \ref{maintheorem: fourier for 1-motives} does not fit neatly into the framework of categorical 1-duality of \cite[Appendix A]{DP08} because integration along ind-schemes requires pushforward with compact support, while integration along affine schemes uses the ordinary pushforward. The issue already appeared in Laumon's work.

If $H$ is trivial, then the transform $\QCoh(G) \to \QCoh(BG^\vee)$ can be written only in terms of ordinary pushforward and pullback, and so the issue above disappears.
In fact, the pairing $G \times BG^\vee \to B\Gm$ induces categorical 1-duality in the sense of \cite[Appendix A]{DP08}.

\begin{theoremalpha}
    (Theorems \ref{theorem: FM for B of G dual}, \ref{theorem: FM for B of G dual is symmetric monoidal})
    \label{maintheorem: fourier for BG^vee}
    Let $R$ be a Noetherian ring admitting a dualizing complex.
    Let $G$ be a flat Mittag-Leffler affine commutative group scheme over $R$.
    Then the Fourier-Mukai transform
    \[ 
        \QCoh(G) \to \QCoh(BG^\vee)
    \]
    induced by the pairing $G \times BG^\vee \to B\Gm$ is an equivalence.
    The functor upgrades to a symmetric monoidal equivalence
    taking convolution in $\QCoh(G)$ to tensor product in $\QCoh(BG^\vee)$.
\end{theoremalpha}

Reversing the roles of $G$ and $G^\vee$, there is also a version of the Fourier-Mukai transform for $BG$ and $G^\vee$.
\begin{theoremalpha}(Theorem \ref{theorem: FM for BG})
    \label{maintheorem: fourier for BG}
    Let $R$ be a Noetherian ring and let $G$ be a flat Mittag-Leffler affine commutative group scheme.
    Then the pairing $G^\vee \times BG \to B\Gm$ induces an equivalence 
    \[ \IndCoh(G^\vee)^{\geq 0} \simeq \QCoh(BG)^{\geq 0}.\]
\end{theoremalpha}
When $G^\vee$ is a formal Lie group, then the equivalence of Theorem \ref{maintheorem: fourier for BG} can be extended to an equivalence of unbounded derived categories (see Theorem \ref{theorem: FM for dual formal Lie group} and Remark \ref{remark: symmetric monoidal structure and BG}).
\subsection{Conventions}

We use $\infty$-categories in the sense of Lurie \cite{Lur09}.
However, the results of this paper are 1-categorical in nature until §\ref{section: categorical 1-duality}.

All of our results are formulated over a base affine scheme (which at some points may be Noetherian or also admit a dualizing complex). However, the results are local over the base, and thus evidently generalize to a general base scheme (or prestack).

\subsection*{Acknowledgments}

The basic idea that flat Mittag-Leffler modules gives an appropriate foundation for infinite-dimensional linear duality is due to Vladimir Drinfeld. 
Theorem \ref{maintheorem: cartier} was independently proved by Akhil Mathew and Vadim Vologodsky.\footnote{Mathew and Vologodsky considered affine group schemes with countably generated projective coordinate ring. Recall that by \cite{RG71}, a countably generated module is projective if and only if it is flat Mittag-Leffler.}
The authors thank 
Justin Campbell,
Vladimir Drinfeld,
Roman Fedorov,
Andreas Hayash,
Akhil Mathew,
Tony Pantev,
Nick Rozenblyum,
and
Vadim Vologodsky
for useful conversations and communications.

In the course of this work, the second-named author was supported by the National Science Foundation under Award No. 2503534.
Any opinions, findings, and conclusions or recommendations expressed in this material are those of the authors and do not necessarily reflect the views of the National Science Foundation.

\section{Linear duality}\label{section: linear duality}

Let $R$ be a commutative ring. In this section, we consider the linear duality for $R$-modules. More precisely, we construct an antiequivalence between the category of flat Mittag-Leffler modules and the category of pro-projective pro-finite Mittag-Leffler pro-modules.

We learned of the notion of a flat Mittag-Leffler module from \cite[§2]{Dri06}.
Our exposition is inspired by \cite[§7.12]{BD91_Hitchin}, with the main results going back to Raynaud and Gruson \cite{RG71}.
There is also an exposition of Raynaud and Gruson's work contained in \cite[§10.86-95]{stacks-project}, which we cite as a textbook reference.

\subsection{Mittag-Leffler modules}

Recall the Mittag-Leffler condition for inverse systems of sets:

\begin{definition}[\cite{stacks-project}, 10.86.1]
\label{def: mittag-leffler}
    If $I$ is a filtered poset and $A= (A_i,\varphi_{ji}:A_j \to A_i)$ is an inverse system of sets indexed by $I$,
    then $A$ is \emph{Mittag-Leffler} if for all $i$ there exists $j \geq i$ such that if $k \geq j$,
    \[ \varphi_{ki}(A_k) = \varphi_{ji}(A_j).\]
\end{definition}

Now let $M$ be an $R$-module. Recall that we can write $M$ as a filtered colimit of finitely presented $R$-modules
\begin{equation}\label{eq:colimit}
    M=\colim_i M_i.
\end{equation}

\begin{definition}[\cite{stacks-project}, 10.88.7]\label{definition: mittag-leffler}
    Given a presentation \eqref{eq:colimit}, 
    the $R$-module $M$ is \emph{Mittag-Leffler} 
    if for all $R$-modules $Q$, the inverse system $\Hom(M_i,Q)$ of sets is Mittag-Leffler.
\end{definition}


\begin{example}
    If $M = \bigoplus_i M_i$, then $M$ is Mittag-Leffler if and only if each $M_i$ is Mittag-Leffler \cite[Lemma 10.89.10]{stacks-project}.
    As the free module of rank 1 is trivially Mittag-Leffler, we conclude that every projective $R$-module is Mittag-Leffler.
\end{example}


We are primarily interested in flat Mittag-Leffler modules, according to Drinfeld's philosophy that these form a good theory of infinite-dimensional vector bundles in algebraic geometry \cite{Dri06}. In particular, we will see that flat Mittag-Leffler modules have a good duality theory.

\begin{remark}
    Raynaud and Gruson introduced the notion of a Mittag-Leffler module in order to study flat descent for projective modules \cite{RG71}. They proved that an $R$-module is projective if and only if it is flat, Mittag-Leffler, and a direct sum of countably generated modules. It is easy to see that this implies flat descent for projectivity.
    
    Note that a finitely generated module is Mittag-Leffler if and only if it is finitely presented.
    In this way, the class of Mittag-Leffler modules is an extension of the class of finitely presented modules.
    Raynaud and Gruson's theorem is then a generalization of the well-known theorem that a flat finitely presented module is projective.
\end{remark}

\subsection{Pro-modules}
Recall the definition of a pro-completion of a category.

\begin{definition} Let $\CC$ be a category; consider the category $\Fun(\CC,\Sets)$ of functors from $\CC$ to the category of sets $\Sets$. Given a cofiltered inverse system
\[\{x_\alpha\}_{\alpha\in I}\]
in $\CC$, its \emph{formal limit} is
\[\lim_\alpha x_\alpha\in\Fun(\CC,\Sets)^{op}:y\mapsto\colim_\alpha\Hom_\CC(x_\alpha,y);\]
the pro-completion $\Pro\CC$ is the full subcategory of $\Fun(\CC,\Sets)^{op}$ spanned by all cofiltered formal limits.
\end{definition}

The category $\CC$ is identified with a full subcategory of $\Pro\CC$ via the Yoneda embedding 
\[h:\CC\to \Fun(\CC,\Sets)^{op}:x\mapsto h^x=\Hom_{\CC}(x,-).\]

Now put $\CC=\Mod_R$ and let $\End_R(\Mod_R)$ be the category of $R$-linear endo-functors of $\Mod_R$. The forgetful functor 
\[\End_R(\Mod_R)\to\Fun(\Mod_R,\Sets)\]
is fully faithful and its essential image contains all formal limits of $R$-modules. Thus, $\Pro(\Mod_R)$ is identified with a full subcategory of $\End_R(\Mod_R)^{op}$; we denote the embedding by
\[h:\ProMod_R\to\End_R(\Mod_R)^{op}:N\mapsto h^N.\] The objects of $\ProMod_R$ are called \emph{pro-$R$-modules}.

For future use, we record the following easy claim:

\begin{lemma} The embedding $\Mod_R\to\ProMod_R$ admits a right adjoint. Explicitly, the right adjoint sends a formal limit $\lim M_\alpha$ to the same limit evaluated in the category of $R$-modules. \hfill\qed 
\end{lemma}

\begin{definition} The \emph{evaluation functor}
\[ev:\ProMod_R\to\Mod_R\]
is the right adjoint of the full embedding
\[\Mod_R\to\ProMod_R.\]
\end{definition}

Consider now several conditions on pro-$R$-modules.

\begin{definition} A pro-module $N$ is \emph{pro-finite}
if there exists a presentation
\[N\simeq\lim N_\alpha\] 
where all modules of the cofiltered inverse system $\{N_\alpha\}$ are finitely generated. We let 
\[\ProMod_R^f\subset\ProMod_R\] be the full subcategory spanned by pro-finite pro-$R$-modules.
\end{definition}

\begin{remark} Equivalently, $\ProMod_R^f$ is
the pro-completion of the category of finitely generated $R$-modules.
\end{remark}

Consider now the Mittag-Leffler condition for inverse systems. Note the following easy claim.

\begin{lemma}\label{lemma:mittag-leffler pro} The following two conditions on a pro-module $N\in\ProMod_R$ are equivalent:
\begin{enumerate}
\item There exists a presentation 
\[N\simeq\lim N_\alpha\]
with surjective structure maps $\varphi_{\alpha\beta}:N_\beta\to N_\alpha$;
\item All presentations
\[N\simeq\lim N_\alpha\]
satisfy the Mittag-Leffler condition: for all $\alpha$, the submodule $\varphi_{\alpha\beta}(N_\beta)\subseteq N_\alpha$ stabilizes as $\beta$ grows.\qed
\end{enumerate}
\end{lemma}

\begin{remark} The lemma holds for inverse systems of sets, but as we need it for $R$-modules, we formulate it in this generality.
\end{remark}

\begin{definition} 
\label{def: pro-mittag-leffer}
A pro-module $N\in\ProMod_R$ is \emph{Mittag-Leffler} if it satisfies the equivalent conditions of Lemma~\ref{lemma:mittag-leffler pro}. We let
\[(\ProMod_R)^{M\!L}\subset\ProMod_R\]
be the full subcategory spanned by Mittag-Leffler pro-$R$-modules.
\end{definition}
\begin{warning}
The terminology of Definition \ref{def: pro-mittag-leffer} conflicts with that of Definition \ref{def: mittag-leffler} when a module is viewed as a pro-module.
We will see in Proposition \ref{prop:linear duality} that these two distinct notions are linearly dual.
We hope that in each instance, it is clear whether Definition \ref{def: mittag-leffler} or Definition \ref{def: pro-mittag-leffer} is being applied.
\end{warning}

Finally, note that for any $N\in\ProMod_R$, the functor
\[h^N:\Mod_R\to\Mod_R\]
is left exact.

\begin{definition} A pro-module $N\in\ProMod_R$ is \emph{pro-projective} if $h^N$ is exact. Let 
\[(\ProMod_R)^{pproj}\subset\ProMod_R\]
be the full subcategory of $\ProMod_R$ spanned by pro-projective pro-$R$-modules.
\end{definition}
\begin{remark}\label{remark: pro-projective conditions}
    By \cite[Proposition 5.2.1]{Bou23}, the following are equivalent for a pro-$R$-module $N = \lim_\alpha N_\alpha$:
    \begin{itemize}
        \item $N$ is pro-projective;
        \item for each $\alpha$, there exists $\beta \geq \alpha$ such that $N_\beta \to N_\alpha$ factors through a projective module;
        \item $N$ is equivalent to a cofiltered formal limit of projective $R$-modules.
    \end{itemize}
    Note also that the condition that $N_\beta \to N_\alpha$ factors through a projective module is equivalent to the induced map $\Ext^1_R(N_\alpha,-) \to \Ext^1_R(N_\beta,-)$ being zero.
    It follows that $\ProMod_R^{pproj}$ is the pro-completion of the category of projective $R$-modules.
\end{remark}

We will also combine the indexes for subcategories of $\ProMod_R$; thus,
\[(\ProMod_R)^{pproj,\ML}=\ProMod_R^{pproj} \cap \ProMod_R^{\ML},\]
and
\[(\ProMod_R^f)^?=\ProMod_R^f\cap(\ProMod_R)^?,\]
where $?$ is a subset of the set $\{pproj,\ML\}$

\subsection{Linear duality}

There is a full embedding of categories 
\begin{align*}
    h: \ProMod_R&\hookrightarrow \End_R(\Mod_R)^{op} \\
    N&\mapsto h^N.
\intertext{On the other hand, tensor product gives a functor}
    t: \Mod_R&\to\End_R(\Mod_R)\\
    M&\mapsto t_M = (M\otimes_R -),
\end{align*}
which is easily seen to be a full embedding as well. 

Let $\Mod_R^{fl} \subset \Mod_R$ be the full subcategory of flat $R$-modules.
\begin{theorem}[\cite{BD91_Hitchin}, Proposition 7.12.6]
\label{thm: intersection of the categories}
\begin{enumerate}
\item Suppose $M\in\Mod_R$. Then an object $N\in\ProMod_R$ such that $h^N\simeq t_M$ exists iff $M\in\Mod_R^{fl}$.

\item Suppose $N\in\ProMod_R$. Then an object $M\in\Mod_R$
such that $h^N\simeq t_M$ exists iff $N\in(\ProMod^f_R)^{pproj}$.

\item Suppose $M\in\Mod_R^{fl}$ and $N\in(\ProMod^f_R)^{pproj}$ are such that $t_M\simeq h^N$.
Then $M$ is Mittag-Leffler if and only if $N$ is.
\end{enumerate}
\end{theorem}

\begin{proof} The proof is sketched in \cite{BD91_Hitchin}; for the sake of completeness we fill in the details.

Proof of \textit{(1)}. Since $h_N$ is left exact, the condition that $M$ is flat is clearly necessary. To prove it is sufficient, we use Lazard's Theorem: if $M$ is flat, then we can write $M$ as a filtered colimit
\begin{equation}\label{eq:free colimit}
M\simeq \colim F_\alpha
\end{equation}
for free finitely generated modules $F_\alpha$. Then $t_M\simeq h^N$ for $N$ given by
\begin{equation}\label{eq:dual limit}
N=\lim F_\alpha^\vee\in\ProMod_R,
\end{equation}
where $F_\alpha^\vee = \Hom_R(F_\alpha,R)$ is the usual linear dual of a finite free module.

Proof of \textit{(2)}. It follows from \textit{(1)} that the condition is necessary: if $t_M\simeq h^N$, then $M$ must be flat and therefore $N$ is given by \eqref{eq:dual limit}. To prove it is sufficient, we use a form of Eilenberg-Watts Theorem: namely, an $R$-linear functor $\Mod_R\to\Mod_R$ is of the form $t_M$ for some $M\in\Mod_R$ if and only if it is right exact and commutes with direct sums. If $N$ is pro-projective, then $h^N$ is exact, and if $N\in\ProMod^f_R$, then $h^N$ commutes with direct sums. 

Proof of \textit{(3)}. Let us represent $M$ as a filtered colimit as in \eqref{eq:free colimit}; then $N$ is given by \eqref{eq:dual limit}. For any $Q\in\Mod_R$, we can identify 
\[\Hom_R(F_\alpha,Q)\simeq F_\alpha^\vee\tensor{R} Q.\]
This inverse system satisfies the Mittag-Leffler condition for all $Q$ if and only if $\{F_\alpha^\vee\}$ satisfies the Mittag-Leffler condition. But by \eqref{eq:dual limit}, this is equivalent to $N$ being Mittag-Leffler.
\end{proof}


\begin{corollary} \label{cor: linear duality equivalence}
There exists an equivalence
\[(\Mod_R^{fl})^{op}\simeq(\ProMod^f_R)^{pproj},\] unique up to a canonical isomorphism. 
\end{corollary}

\begin{definition} We refer to the equivalence as \emph{linear duality} and denote both it and its inverse by $M\mapsto M^\vee$.
\end{definition}

Theorem~\ref{thm: intersection of the categories} implies basic properties of linear duality, which we record below.

First, let $\Mod_R^{f,proj}\subset\Mod_R$ be the category of projective $R$-modules of finite rank. Then, on the one hand, 
\[\Mod_R^{f,proj}\subset\Mod_R^{fl}.
\]
On the other hand, consider the embedding
\[\Mod_R^{f,proj}\hookrightarrow\Mod_R\hookrightarrow\ProMod_R.\]
Clearly, the image of $\Mod_R^{f,proj}$ is contained in
$(\ProMod_R^f)^{pproj}$.

\begin{proposition} Suppose $M\in\Mod_R^{f,proj}$.  
\begin{enumerate}
\item Consider $M$ as an object of $\Mod_R^{fl}$. Then its linear dual $M^\vee$ is identified with the dual $R$-module $\Hom_R(M,R)$.

\item Dually, consider $M$ as an object of $(\ProMod_R^f)^{pproj}$. Then its linear dual is also identified with the dual $R$-module.\qed
\end{enumerate}
\end{proposition}


\begin{remark}
It is easy to see that 
$\Mod_R^{f,proj}$ is identified with the intersection
\[\Mod_R\cap(\ProMod_R^f)^{pproj}.\]
In particular, we see that the notation $M^\vee$ for the dual object is unambiguous, even though we use it in three settings: for flat $R$-modules, for pro-projective pro-finite pro-$R$-modules, and for projective modules of finite rank. 
\end{remark}

\begin{proposition}\label{prop:linear duality}
\begin{enumerate}
\item If $M\in\Mod_R^{fl}$ is represented as a filtered colimit of free modules of finite rank
\[M\simeq\colim F_\alpha,\] we obtain an isomorphism
\[M^\vee\simeq\lim F_\alpha^\vee\in(\ProMod_R^f)^{pproj}.\]

\item Dually, if $N\in(\ProMod_R^f)^{pproj}$ is represented as a cofiltered limit of free modules of finite rank
\[N\simeq\lim F_\alpha,\]
we obtain an isomorphism
\[N^\vee \simeq\colim F_\alpha^\vee.\]

\item A flat module $M\in\Mod_R^{fl}$ is Mittag-Leffler if and only if its dual $M^\vee\in(\ProMod_R^f)^{pproj}$ is Mittag-Leffler.
\end{enumerate}
\qed
\end{proposition}

Finally, note that in statements (1) and (2) of Proposition~\ref{prop:linear duality}, representation always exists. For (1), this is Lazard's Theorem, while for (2), we have the following dual statement.

\begin{proposition} \label{prop:coLazard}
Suppose $N\in\ProMod_R$. Then $N\in(\ProMod_R^f)^{pproj}$ if and only if $N$ can be represented as a cofiltered limit of free modules of finite rank.
\end{proposition}
\begin{proof}
The `if' direction follows from definitions. For the `only if'
direction, consider $M=N^\vee\in\Mod_R^{fl}$, represent it as a filtered colimit of free modules of finite rank via Lazard's Theorem, and apply Proposition~\ref{prop:linear duality}(1).
\end{proof}

\begin{remark}\label{remark: duality by continuity}
In other words, we can use Lazard's Theorem to identify $\Mod_R^{fl}$ with the ind-completion of the category $\Mod_R^{f,proj}$. (The notion of ind-completion is opposite to that of pro-completion, so $(\Mod_R^{fl})^{op}\simeq\Pro((\Mod_R^{f,proj})^{op})$.) Similarly, Proposition~\ref{prop:coLazard} identifies $(\ProMod_R^f)^{pproj}$ with $\Pro(\Mod_R^{f,proj})$. Now the linear duality 
\[(\Mod_R^{fl})^{op}\simeq(\ProMod_R^f)^{pproj}\]
can be obtained from the ``usual" duality
\[(\Mod_R^{f,proj})^{op}\simeq\Mod_R^{f,proj}\]
by continuity.
\end{remark}

\begin{remark}
    Over a field $k$, $\Pro(\mathrm{Vect}_k^f)$ may be identified with linearly compact topological vector spaces over $k$. However, for general $R$, linearly compact topological modules correspond to a \emph{proper} subcategory of $\ProMod_R^f$.
    For if $\{M_\alpha\}_{\alpha \in I}$ is a Mittag-Leffler pro-system, it is not necessarily the case that $\lim_\beta M_\beta \to M_\alpha$ is surjective; such pro-systems are known as \emph{strictly Mittag-Leffler}. 
    See \cite[(7.12.16)]{BD91_Hitchin} and references therein for an example of a flat Mittag-Leffler module which is not strictly Mittag-Leffler.
\end{remark}

\subsection{Linear duality and monoidal structure}

The tensor product of $R$-modules equips $\Mod_R$ with a symmetric monoidal structure. 
The symmetric monoidal structure extends to a symmetric monoidal structure on $\ProMod_R$
uniquely determined by compatibility with cofiltered limits. Explicitly, the tensor product on $\ProMod_R$ is given by the formula
\[(\lim_\alpha M_\alpha)\otimes(\lim_\beta N_\beta)=\lim_{\alpha,\beta}(M_\alpha\otimes N_\beta). \]

Consider now the linear duality. On finite projective modules, it is an equivalence
\begin{equation}\label{eq:linear duality proj}
(\Mod_R^{f,proj})^{op}\to\Mod_R^{f,proj};
\end{equation}
this equivalence naturally upgrades to an equivalence of symmetric monoidal categories.

\begin{proposition} \label{prop:linear duality monoidal}
The symmetric monoidal structure on \eqref{eq:linear duality proj} has a unique (up to a canonical isomorphism) extension to a symmetric monoidal structure on the linear duality 
\[(\Mod_R^{fl})^{op}\simeq(\ProMod^f_R)^{pproj}.\]
\end{proposition}
\begin{proof}
As in Remark \ref{remark: duality by continuity}, we can identify $\Mod_R^{fl}$ with the ind-completion of $\Mod_R^{f,proj}$ and $(\ProMod_R^f)^{pproj}$ with its pro-completion. The symmetric monoidal structure on the two categories is obtained from that on $\Mod_R^{f,proj}$ by ind-extension and pro-extension, respectively. Since the extensions are unique up to unique isomorphisms, they are compatible with the linear duality. 
\end{proof}

From now on, we consider the linear duality as a symmetric monoidal functor via the structure of Proposition~\ref{prop:linear duality monoidal}.

Given a finite projective $R$-module $F$ and an $R$-module $L$, there is a isomorphism
\[ \Hom_R(F,L) \cong F^\vee \otimes L \]
natural in $F$ and $L$.
\begin{proposition} \label{prop:unit}
Given $M\in\Mod_R^{fl}$ and $L\in\Mod_R$, there exists a natural isomorphism of $R$-modules
\[\Hom_R(M,L) \cong ev(M^\vee\otimes L)\]
uniquely extending the natural isomorphism for $M\in\Mod_R^{f,proj}$.
\end{proposition}
\begin{proof} 
Indeed, both functors send filtered colimits in $M$ into limits.
\end{proof}

\begin{remark}\label{remark: existence of unit for duality}
In particular, setting $M=L$, we obtain a canonical unit element $\unit_M\in ev(M^\vee\otimes M)$ for any $M\in\Mod_R^{fl}$. Note that we can reinterpret it as a homomorphism of pro-modules
\[\unit_M:R\to M^\vee\otimes M.\]
The morphism $\unit_M$ determines the isomorphism of Proposition~\ref{prop:unit} by composition with homomorphisms $M\to L$. 

Note that there is no natural counit element
\[M^\vee\otimes M \overset{?}\to R\]
unless $M\in\Mod_R^{f,proj}$. This lack of symmetry between units and counits is caused by our choice to work in the category $\ProMod_R$: thus, given a colimit and a limit
\begin{align*}
M&=\colim M_\alpha\in\Mod_R^{fl}& M_\alpha&\in\Mod_R^{f,proj}\\
N&=\lim N_\beta\in(\ProMod_R^f)^{pproj}& N_\beta&\in\Mod_R^{f,proj},
\end{align*}
their product is given by
\[N\otimes M=\lim_\beta\colim_\alpha N_\beta\otimes M_\alpha.\]
Essentially, the unit is defined because we work with pro-ind-objects of $\Mod_R^{f,proj}$, while counit would require working with ind-pro-objects.  
\end{remark}

\begin{remark}\label{remark: using the unit the wrong way}
    If $M \in \Mod_R^{fl}$, then by definition of linear duality there is an isomorphism
    \[ \Hom_{\ProMod_R}(M^\vee, L) \to M \otimes L\]
    natural in $L \in \Mod_R$.
    This isomorphism is also induced by the \emph{same} unit $\unit_M : R \to M^\vee \otimes M$ as in Remark \ref{remark: existence of unit for duality},
    i.e. the isomorphism sends $f: M^\vee \to L$ to $(f \otimes 1)\unit_M \in M \otimes L$.
    This can be checked by writing an explicit formula for $\unit_M$ in terms of a presentation $M = \colim_\alpha F_\alpha$ as a filtered colimit of finite free modules.
\end{remark}

While a general flat $R$-module $M$ is not dualizable, the following weaker statement concerning maps from a flat module to the dual of a flat module holds:

\begin{proposition}\label{prop: mapping flat into coflat and bilinear}
    Let $M, N \in \Mod_R^{fl}$. Then there exist natural isomorphisms of $R$-modules
    \[ 
        \Hom_{\ProMod_R}(M,N^\vee) \cong \Hom_R(M \otimes N, R) \cong ev( (M \otimes N)^\vee) \cong \Hom_{\ProMod_R}(N, M^\vee)
    \]
    uniquely extending the natural isomorphisms when $M,N \in \Mod_R^{f,proj}$
\end{proposition}
\begin{proof}
    All three functors send filtered colimits in $M$ and $N$ into limits.
\end{proof}

\subsection{Pro-algebras}\label{subsection: proalgebras}

Once we have equipped the category $\ProMod_R$ with a symmetric monoidal structure, we can consider (associative unital) algebras in it. Denote by
\[\Alg(\ProMod_R)\]
the category of such algebras.

On the other hand, we can consider the category $\Alg_R$ of algebras in $\Mod_R$, and take its pro-completion $\Pro(\Alg_R)$. The natural functor
\[\Pro\Alg_R\to\Alg(\ProMod_R)\]
is fully faithful; its essential image consists of algebras in pro-modules that can be represented as a cofiltered limit of algebras in modules.

\begin{theorem}\label{theorem: algebras in pro}
Suppose $A\in\Alg(\ProMod_R)$ and that, as a pro-$R$-module, $A$ satisfies the Mittag-Leffler condition. 
\begin{enumerate}
\item $A$ can be represented as a cofiltered limit
\[A=\lim A_\alpha,\]
of $R$-algebras $A_\alpha$ with surjective transition maps. 

\item If $A$ is pro-finite as a pro-$R$-module, then we can choose $A_\alpha$ to be finitely generated $R$-modules for all $\alpha$. 

\item If $A$ is commutative (resp. commutative and pro-finite), then we can choose $A_\alpha$ to be commutative (resp. commutative and finitely generated) for all $\alpha$.

\end{enumerate}
\end{theorem}

The theorem is a generalization of a theorem of Dieudonné that over a field $k$, a linearly compact commutative algebra is the inverse limit of finite-dimensional $k$-algebras \cite[p.12, Proposition 1]{Die73}.

\begin{proof}
Let us represent the pro-$R$-module $A$ as a cofiltered limit of a surjective system of $R$-modules
\[A=\lim M_\alpha.\]
Each $M_\alpha$ has a maximal quotient $M_\alpha\twoheadrightarrow A_\alpha$ such that the projection $A\to M_\alpha$ induces an algebra structure on $A_\alpha$. Explicitly, $A_\alpha$ can be constructed as follows:
The algebra structure on $A$ induces a map \[\mu^{(3)}:A^{\otimes 3}\to A;\]
therefore, for $\beta$ large enough, we obtain a map
\[\mu^{(3)}_{\alpha\beta}:M_\beta^{\otimes 3}\to M_\alpha.\]
Put \[A_\alpha:=M_\alpha/\mu^{(3)}_{\alpha\beta}(M_\beta\otimes\ker(\varphi_{\alpha\beta})\otimes M_\beta),\]
where $\varphi_{\alpha\beta}:M_\beta\to M_\alpha$ is the structure map. Since the maps $\mu^{(3)}_{\alpha\beta}$ depend on $\beta$ in a compatible way, the definition of $A_\alpha$ does not depend on $\beta$, and thus the product $M_\beta \otimes M_\beta \to M_\alpha$ descends to $A_\alpha \otimes A_\alpha \to A_\alpha$. 

The algebras $A_\alpha$ form a cofiltered system of algebras and the quotients $M_\alpha \twoheadrightarrow A_\alpha$ form a homomorphism 
\[A=\lim M_\alpha\to\lim A_\alpha\]
in $\Alg(\ProMod_R)$.
Let us show that this homomorphism is an isomorphism. It suffices to show that the underlying map of pro-modules is an isomorphism.
Equivalently, we need to show that, for all $\alpha$, there exists $\beta\ge\alpha$ such that the projection $M_\beta \to M_\alpha$ factors through $M_\beta \to A_\beta$. 
It suffices to take $\beta$ large enough so that the map $\mu^{(3)}_{\alpha\beta}$ is defined. This follows from the diagram 
\[\begin{tikzcd}
	{M_\gamma^{\otimes 3}} & {M_\beta} \\
	{M_\beta^{\otimes 3}} & {M_\alpha,}
	\arrow["{\mu^{(3)}_{\beta\gamma}}", from=1-1, to=1-2]
	\arrow["{(\varphi_{\beta\gamma})^{\otimes 3}}"', from=1-1, to=2-1]
	\arrow["{\mu^{(3)}_{\alpha\gamma}}"{description}, from=1-1, to=2-2]
	\arrow["{\varphi_{\alpha\beta}}", from=1-2, to=2-2]
	\arrow["{\mu^{(3)}_{\alpha\beta}}"', from=2-1, to=2-2]
\end{tikzcd}
\]
with $\gamma\ge\beta$ large enough. Indeed, since
\[A_\beta=M_\beta/\mu^{(3)}_{\beta\gamma}(M_\gamma\otimes\ker(\varphi_{\beta\gamma})\otimes M_\gamma),\]
we need to show that
\[\varphi_{\alpha\beta}(\mu^{(3)}_{\beta\gamma}(M_\gamma\otimes\ker(\varphi_{\beta\gamma})\otimes M_\gamma))=0.\]
But
\begin{multline*}
\varphi_{\alpha\beta}(\mu^{(3)}_{\beta\gamma}(M_\gamma\otimes\ker(\varphi_{\beta\gamma})\otimes M_\gamma))=\mu^{(3)}_{\alpha\beta}(\varphi_{\beta\gamma}^{\otimes 3}(M_\gamma\otimes\ker(\varphi_{\beta\gamma})\otimes M_\gamma))\\
=\mu^{(3)}_{\alpha\beta}(\varphi_{\beta\gamma}^{\otimes 3}(M_\beta\otimes 0\otimes M_\beta))=0.
\end{multline*}

Clearly, if $A$ is commutative (resp. pro-finite), then $A_\alpha$ is commutative (resp. finitely generated).
\end{proof}

\begin{remark}
    Over a field, Theorem \ref{theorem: algebras in pro} 
    is the linear dual of the \emph{Fundamental Theorem of Coalgebras}: a coalgebra over a field $k$ is the union of its finite-dimensional sub-coalgebras \cite[Theorem 2.2.1]{Swe69}.
    However, when $R$ is not a field, we do not know whether every flat Mittag-Leffler $R$-coalgebra is a filtered colimit of finitely generated coalgebras.
    Note that this statement does not follow from Theorem \ref{theorem: algebras in pro}.
    One positive result in this direction is that if $C$ is a flat Mittag-Leffler $R$-coalgebra,
    then every $C$-comodule is a union of finitely generated subcomodules \cite[Theorem 3.16]{BW03}.
\end{remark}

\begin{corollary} \label{corollary: ML proalgebras}
The embedding $\Pro(\Alg_R)\to\Alg(\ProMod_R)$ induces an equivalence between the following categories:
\begin{enumerate}
\item The category $\Pro(\Alg_R)^{\ML}$ of pro-$R$-algebras satisfying the Mittag-Leffler condition
and the category $\Alg((\ProMod_R)^{\ML})$ of algebras in pro-$R$-modules satisfying the Mittag-Leffler condition,

\item The category $\Pro(\Alg_R^f)^{\ML}$ of pro-finite pro-$R$-algebras satisfying the Mittag-Leffler condition and the category $\Alg((\ProMod_R)^{f,\ML})$ of algebras in pro-finite pro-$R$-modules satisfying the Mittag-Leffler condition,
\end{enumerate}
as well as in versions of these categories using commutative algebras.
\end{corollary}

Let $\CAlg_R$ be the category of commutative (unital) $R$-algebras. Recall the following definition:
\begin{definition}\label{def:ind-schemes}
An \emph{ind-affine $R$-ind-scheme} is a functor $\CAlg_R \to \Sets$ of the form $\colim_{\alpha\in I} X_\alpha$ where $I$ is filtered, $X_\alpha \to \Spec R$ is an affine $R$-scheme, 
and the maps $X_\alpha \to X_\beta$ are closed embeddings. If we can choose $X_\alpha$ to be finite over $R$, the ind-scheme is \emph{ind-finite}.
\end{definition}

Recall that the pro-completion $\Pro(\CAlg_R)$ is by definition a subcategory
\[\Pro(\CAlg_R)\subset\Fun(\CAlg_R,\Sets)^{op}.\]

\begin{proposition}\label{prop:ind-schemes} 
Consider a functor $h:\CAlg_R\to\Sets$. 
\begin{enumerate}
    \item $h$ is representable by an ind-affine $R$-ind-scheme if and only if \[h\in(\Pro\CAlg_R)^{\ML}.\]

    \item $h$ is representable by an ind-finite $R$-ind-scheme if and only if \[h\in(\Pro\CAlg_R^f)^{\ML}.\]
\end{enumerate}
\end{proposition}
\begin{proof}
It suffices to rewrite affine $R$-schemes $X_\alpha$ as $\Spec A_\alpha$ for $R$-algebras $A_\alpha$
in Definition~\ref{def:ind-schemes}.
\end{proof}

Combining Corollary~\ref{corollary: ML proalgebras} and Proposition~\ref{prop:ind-schemes}, we obtain an antiequivalence between
$\CAlg((\ProMod_R)^{\ML})$ and the category of ind-affine $R$-ind-schemes, which we denote by $\SpInd$.
Explicitly, given $A\in\CAlg((\ProMod_R)^{\ML})$, we have
\[\SpInd(A):\CAlg_R\to\Sets:S\mapsto\Hom_{\CAlg(\ProMod_R)}(A,S).\] 
By Theorem \ref{theorem: algebras in pro}, $A$ is an inverse limit of $R$-algebras $A_\alpha$ with surjective transition maps, so
\[ \SpInd(A) = \colim_\alpha \Spec(A_\alpha).\]

Moreover, $\SpInd(A)$ is ind-finite
if and only if $A\in\CAlg((\ProMod_R^f)^{\ML}).$
\begin{remark}
    If $A$ is an admissible topological ring, that is, $A$ is the completion of $(B,I)$ in $I$-adic topology where $B$ is a ring and $I \subseteq B$, then $\SpInd(A)$ agrees with $\mathrm{Spf}(A)$ in the sense of EGA.
\end{remark}
\begin{definition}\label{definition: coflat}
    An ind-affine $R$-ind-scheme $\SpInd A \to \Spec R$ is \emph{coflat} if $A$ is pro-projective as pro-$R$-module.
\end{definition}

\begin{remark}
    Since filtered colimits commute with finite limits in $\Sets$, the fiber product of ind-affine $R$-ind-schemes is computed by the tensor product of commutative algebras in $(\ProMod_R)^{\ML}$.
\end{remark}

\begin{remark}
    Recall that an ind-scheme is said to be $\aleph_0$ if it is equivalent to a colimit 
    of schemes under closed embeddings indexed by $\mathbb N$
    \cite[§5.2.1]{GR14_indschemes}.
    The proof of Theorem \ref{theorem: algebras in pro} shows that $\SpInd(A)\to \Spec R$ is $\aleph_0$ if and only if $A$ is $\aleph_0$ as a pro-$R$-module. If $A \in (\Pro\CAlg_R^f)^{\ML}$, then $A$ is $\aleph_0$ if and only if $A^\vee \in \Mod_R$ is a countably generated flat Mittag-Leffler $R$-module, thus projective by the theorem of Raynaud and Gruson. 
\end{remark}

\subsection{Pro-projective and pro-flat modules}

It is a well-known theorem in homological algebra that a finitely presented module is flat if and only if it is projective.
There is also a version of this theorem for pro-modules with an appropriate version of flatness for pro-modules.
This subsection will not be used in the sequel but may be of independent interest.
\begin{definition}
    A pro-module $N = \lim_\alpha N_\alpha\in \ProMod_R$
    is \emph{pro-flat}
    if for all $\alpha$ there exists $\beta \geq \alpha$ such that the induced natural transformation of functors
    \[ \Tor_1^R(N_\beta,-) \to \Tor_1^R(N_\alpha,-)\]
    is zero.

    An ind-affine ind-scheme $\SpInd(A) \to \Spec R$ is \emph{flat} if $A$ is a pro-flat pro-$R$-module.
\end{definition}

\begin{lemma}[\cite{Bou23}, Proposition 5.4.7]
    Suppose $N = \lim_\alpha N_\alpha$ is a pro-$R$-module where each $N_\alpha$ is finitely presented.
    Then $N$ is pro-flat if and only if $N$ is pro-projective.
\end{lemma}

When $R$ is Noetherian, there is no difference between finitely generated and finitely presented, yielding:

\begin{corollary}\label{cor: noetherian pro-flat vs pro-projective}
    If $R$ is Noetherian, then $N \in \ProMod_R^{f}$ is pro-flat if and only if $N$ is pro-projective.
\end{corollary}


\section{0-duality}

In this section, we relate Cartier duality (for sheaves of groups) to linear duality (between modules and pro-modules). The main result is Theorem~\ref{maintheorem: cartier} from the introduction, proved in Corollary \ref{corollary: cartier duality is an antiequivalence}.

\subsection{Duality of Hopf algebras and Cartier duality}
Let $\Shv$ be the category of sheaves of sets on the big fppf site of $\Spec R$. We identify the category of $R$-indschemes with a full subcategory of $\Shv$. 

Denote by $\Shvab$ the category of abelian group objects in $\Shv$. Enter the hero of the story:

\begin{definition}
    The \emph{Cartier dual} of $G\in\Shvab$ is 
    \[G^\vee = \Homs(G,\mathbb G_m)\in\Shvab.\]
\end{definition}

It is classical that Cartier duality induces an auto-equivalence of the category of finite locally free group schemes over any base \cite[Exposé VIIA, §3.3]{SGA3I}.
It is also classical that if $k$ is a field, then Cartier duality induces an equivalence of categories between affine commutative $k$-group schemes and commutative group ind-finite ind-schemes \cite{Car62, Die73}.
In both cases, the duality is implemented by linear duality of the underlying Hopf algebra. In the first case, this is ordinary linear duality of finite projective modules; in the second, this is linear duality between vector spaces and linearly compact topological vector spaces.

Proposition \ref{prop:linear duality monoidal} states that linear duality between flat Mittag-Leffler modules and pro-projective Mittag-Leffler pro-modules is symmetric monoidal.
Thus, the linear dual of a Hopf algebra is again a Hopf algebra.
\begin{corollary} 
\label{corollary: duality for Hopf algebras}
The symmetric monoidal antiequivalence of categories
\[\Mod_R^{fl, \ML}\simeq^{op}(\ProMod^{f}_R)^{pproj,\ML}\]
induces antiequivalence between categories of Hopf algebras in the two categories. The antiequivalence sends commutative Hopf algebras to cocommutative Hopf algebras and vice versa. \qed
\end{corollary}

Thus, if $A$ is a commutative cocommutative Hopf algebra in $\Mod_R^{fl,\ML}$, $A^\vee$ is a commutative cocommutative Hopf algebra in $\ProMod_R^{f,pp,\ML}$, and thus $\SpInd(A^\vee)$ is a commutative group ind-finite ind-scheme.

\begin{definition}
    An affine $R$-scheme $\Spec A\to\Spec R$ is \emph{flat Mittag-Leffler} if $A$ is a flat Mittag-Leffler $R$-module\footnote{Since the property of a module being flat Mittag-Leffler is preserved under arbitrary base change and is local in the fpqc topology, the notion naturally extends to affine morphisms of fpqc stacks.}. 
\end{definition}

\begin{remark}
In their seminal paper \cite{RG71} on Mittag-Leffler modules, Raynaud and Gruson gave criteria for a morphism of rings $R \to R'$ to be flat Mittag-Leffler. Here is one sufficient condition:
if $R\to R'$ be a smooth morphism of rings with geometrically integral fibers, then $R'$ is a projective $R$-module.\cite[p. 19 (3.3.1)]{RG71} 
\end{remark}

If $G =\Spec A \to \Spec R$ is a flat Mittag-Leffler affine commutative group scheme, then $A$ is a flat Mittag-Leffler Hopf algebra over $R$, and Corollary \ref{corollary: duality for Hopf algebras} applies. The following theorem shows that the dual Hopf algebra $A^\vee$ is naturally the pro-coordinate ring of the Cartier dual of $G$.

\begin{theorem}\label{theorem: perfect pairing for flat ML group}
    Let $G = \Spec A \to \Spec R$ be a flat Mittag-Leffler commutative group scheme over $\Spec R$. 
    Then there is a perfect bilinear pairing $G \times \SpInd(A^\vee) \to \mathbb G_m$; that is, a pairing that induces isomorphisms of Cartier duals $ \SpInd(A^\vee)\cong G^\vee$ and $G\cong \SpInd(A^\vee)^\vee$ in $\Shvab$.
\end{theorem}
\begin{proof}
Let $G = \Spec A$ be a flat Mittag-Leffler commutative group $R$-scheme, so that $A$ is a flat Mittag-Leffler Hopf algebra.
By Theorem~\ref{theorem: algebras in pro}, a morphism
\[\Spec(A)\times\SpInd A^\vee\to\Gm\]
is the same as an element $g\in ev(A\otimes A^\vee)^\times$.


Consider $g=\unit_A\in ev(A\otimes A^\vee)$. Linear duality implies that $\unit_A$ is grouplike with respect to each tensor factor. That is, if $\nabla: A \otimes A \to A$ is the product, $\eta: R \to A$ is the unit, $\Delta: A \to A \otimes A$ is the coproduct, and $\epsilon: A \to R$ is the counit of the Hopf algebra $A$, then  
\begin{equation}\label{eq: hopf algebra dualization and the unit}
    \begin{aligned}
        (1 \otimes \Delta) \circ \unit_{A} &= ( \Delta^\vee\otimes 1) \circ \unit_{A \otimes A} \\
    (1 \otimes \eta) \circ \unit_R &= (\eta^\vee \otimes 1) \circ \unit_{A} \\
    (1 \otimes \nabla) \circ \unit_{A \otimes A} &= (\nabla^\vee \otimes 1) \circ \unit_{A} \\
    (1 \otimes \epsilon )\circ \unit_{A} &= (\epsilon^\vee \otimes 1) \circ \unit_R.
    \end{aligned}
\end{equation}
It follows that $\unit_A$ is invertible, thus giving a bilinear pairing 
\begin{equation}\label{eq:pairing}G\times\SpInd(A^\vee)\to\Gm.
\end{equation}

Let us show that the pairing \eqref{eq:pairing} is perfect. Let $B$ be an $R$-algebra.
By Remark \ref{remark: existence of unit for duality}, 
composing with $\unit_M$ induces the isomorphism $\Hom_R(M,-) \to ev(M^\vee \otimes -)$.
The equations \eqref{eq: hopf algebra dualization and the unit} further imply that $f \mapsto (1 \otimes f)\unit_{A}$
is a bijection
\begin{align*} 
(\Spec A)(B) &= \{f: A \to B \text{ algebra homomorphism}\} \\
&\cong \{g \in ev(A^\vee \otimes B) \text{ grouplike}\} \\
&= \Homs(\SpInd(A^\vee),\mathbb G_m)(B),
\end{align*}
where for the last equality we have again used Theorem~\ref{theorem: algebras in pro}.
Similary, by Remark \ref{remark: using the unit the wrong way}, 
composing with $\unit_M$ induces the isomorphism $\Hom_{\ProMod_R}(M^\vee,-) \to M \otimes -$ of functors on $\Mod_R$,
and $f \mapsto (f \otimes 1)\unit_{A}$ is a bijection
\begin{align*} 
    (\SpInd A^\vee)(B) &= 
    \{f : A^\vee \to B \text{ algebra homomorphism} \} \\
    &\cong \{g \in A \otimes B \text{ grouplike}\} \\
    &= \Homs(\Spec A,\Gm)(B).
\end{align*}
Thus, the pairing \eqref{eq:pairing} induces isomorphisms $(\Spec A)^\vee \cong \SpInd A^\vee$ and $\Spec A \cong (\SpInd A^\vee)^\vee$, as desired.
\end{proof}

\begin{corollary}\label{corollary: cartier duality is an antiequivalence}
    Cartier duality induces an involutive antiequivalence 
    between flat Mittag-Leffler affine commutative group schemes
    and coflat commutative group ind-finite ind-schemes.
\end{corollary}

\begin{remark}\label{remark: derived Cartier duality?}
    In this paper, we have formulated Cartier duality only for classical commutative group schemes.
    The theory of linear duality for flat modules also makes sense in the derived context, as Lazard's theorem holds for connective modules over connective $\Einfty$-rings \cite[Theorem 7.2.2.15]{lurieha}.
    For a connective $\Einfty$-ring $R$, say that $N \in \Mod_R^{\leq 0}$ is \emph{flat Mittag-Leffler} if $N$ is flat as an $R$-module and $\pi_0(N)$ is Mittag-Leffler as a $\pi_0(R)$-module.
    Suppose now $R$ is eventually coconnective.
    Proposition \ref{prop: derived ind-scheme} below yields an antiequivalence between flat Mittag-Leffler commutative group schemes over $R$ and coflat ind-finite commutative group ind-schemes over $R$.
    Here ``commutative group'' means $\Einfty$-grouplike object, as in Lurie's theory of spectral abelian varieties \cite{Lur17}.

    This is not entirely satisfactory if one wants to work with simplicial commutative rings instead of connective $\Einfty$-rings. The difference is only felt away from characteristic zero. In the context of simplicial commutative rings, an $\Einfty$-grouplike object may not be the ``correct'' notion of derived group.
    A better notion may be sheaves with multiplicative polynomial transfers, as in work in progress by Kubrak, Mathew, Raksit, and Zavyalov. It would be interesting to know if Corollary~\ref{corollary: cartier duality is an antiequivalence} holds in this setting.
\end{remark}

\begin{example}
    Without the Mittag-Leffler condition, the Cartier dual of a flat commutative group scheme need not be representable by an ind-scheme.

    Let $R=\C[t]$ and let $A$ be the $R$-algebra $\Sym_{R}(\C[t,t^{-1}]\cdot u)= \C[t] \oplus u\C[t,t^{-1}][u].$
    Then $A$ is a commutative cocommutative Hopf algebra with coproduct defined by $\Delta u = u \otimes 1 + 1 \otimes u$ and counit defined by $\epsilon(u) = 0$. Thus $G=\Spec(A)$ is a flat affine commutative group scheme over $R$. The functor of points of $G$ sends an $R$-algebra $B$ to its colocalization  
    \[G(B)=\lim(\dots B \to B \to B),\]
    where all the maps are multiplication by $t$. 
    Let $G^\vee = \Homs(G,\mathbb G_m)$ be the Cartier dual of $G$. Then
    \[G^\vee(B)=\{g\in B\otimes_{\C[t]}A \mid g\text{ is grouplike}\}\cong\{x\in B\otimes_{\C[t]}\C[t,t^{-1}] \mid x \text{ is nilpotent}\},\]
    where the last isomorphism is given by $g=\exp(ux)$.
    
    We show that $G^\vee$ is not an ind-scheme. Indeed, suppose that $X = \colim_\alpha X_\alpha$ is an ind-scheme over $\Spec R$. Let $x \in X(B)$ be a $B$-point of $X$. The tangent space to $X$ at $x$ is the functor $T_{X,x}: \Mod_B \to \Sets$, 
    \[ T_{X,x}(M) = X(B\oplus M) \times_{X(B)} \{x\}.\]
    If $X_\alpha$ is a scheme, then $T_{X_\alpha,x} = \Hom_B(x^*\Omega_{X_\alpha/R},-)$,
    and if $X_\alpha \to X_\beta$ is a closed immersion,
    then the map $x^*\Omega_{X_\beta/R} \to x^*\Omega_{X_\alpha/R}$ is surjective.
    Thus, if $X = \colim_\alpha X_\alpha$ is an ind-scheme,
    then $T_{X,x}$ is a functor of the form $h^N$ for $N \in \ProMod_B^{\ML}$.\footnote{This is a slight generalization of Beilinson and Drinfeld's observation that the tangent sheaf to a formally smooth ind-scheme of finite type over a field is flat Mittag-Leffler \cite[Proposition 7.12.12]{BD91_Hitchin}.}   
    
    The tangent space of $G^\vee$ at the identity section is the functor $t_{\C[t,t^{-1}]}: M \mapsto M \otimes_{\C[t]} \mathbb{C}[t,t^{-1}]$.
    If $G^\vee$ were an ind-scheme,
    then the functor $t_{\C[t,t^{-1}]}$ would be isomorphic to $h^N$ for a Mittag-Leffler pro-module $N$.
    If so, Theorem \ref{thm: intersection of the categories} implies $\C[t,t^{-1}]$ is a flat Mittag-Leffler $\C[t]$-module.
    But $\C[t,t^{-1}]$ is not Mittag-Leffler,
    as $\C[t,t^{-1}] = \cup_i t^{-i}\C[t]$ and the dual inverse system $\{t^i\C[t]\}_{i \geq 0}$ does not stabilize.
    Thus $G^\vee$ is not an ind-scheme.
\end{example}

\subsection{Extensions of flat Mittag-Leffler group schemes}

The following lemma is a transitivity property for Mittag-Leffler morphisms.

\begin{lemma}[\cite{stacks-project}, Lemma 10.89.11]\label{lemma: mittag-leffler and restriction}
    Let $R'$ be a Mittag-Leffler $R$-algebra and $M$ be a flat Mittag-Leffler $R'$-module.
    Then $M$ is a Mittag-Leffler $R$-module.
\end{lemma}

In particular, the composition of two flat Mittag-Leffler morphisms of rings is again flat Mittag-Leffler.
From this transitivity and Raynaud and Gruson's descent theorem, we obtain that flat Mittag-Leffler commutative group schemes are closed under extensions in fppf sheaves:

\begin{lemma}\label{lemma: extension of ML groups}
    Let $G$ and $H$ be affine commutative group schemes over $\Spec R$
    Let 
    \[ 0 \to H \to E \to G \to 0\]
    be an extension of fppf sheaves of abelian groups on $\Spec R$.
    \begin{enumerate}
        \item $E$ is representable by an affine group scheme over $\Spec R$.
        \item Suppose $H\to \Spec R$ and $G \to \Spec R$ are both flat Mittag-Leffler. Then $E \to \Spec R$ is also flat Mittag-Leffler.
    \end{enumerate}
\end{lemma}
\begin{proof}
    \begin{enumerate}
        \item This is a standard exercise in flat descent; see for example \cite[Proposition 17.4]{Oor66}.
        \item First, we show that the morphism $E \to G$ is flat Mittag-Leffler.
        Since $E \to G$ is a surjection of fppf sheaves, $E \to G$ is faithfully flat.
        The base change of $E \to G$ along $E \to G$ is $E \times_G E \to E$. Thus $E \times_G E$ is identified with $H \times_{\Spec R} E$ as an $E$-scheme via the diagonal section.
        Both flat and Mittag-Leffler morphisms ascend under base change \cite[Lemma 10.94.1]{stacks-project}. Thus $E \times_G E \cong H \times_{\Spec R} E \to E$ is flat Mittag-Leffler.
        As being flat Mittag-Leffler satisfies fpqc descent \cite[Lemma 10.95.1]{stacks-project}, $E \to G$ is flat Mittag-Leffler.

        Now $E \to G$ is flat Mittag-Leffler and $G \to \Spec R$ is flat Mittag-Leffler, so by Lemma \ref{lemma: mittag-leffler and restriction}, $E \to \Spec R$ is flat Mittag-Leffler. 
    \end{enumerate}
\end{proof}

\section{Geometric 1-duality}\label{section: geometric 1-dual}

In this section, we study the \emph{1-duality} functor
\[ (-)^D= \Hom(-,B\Gm),\] 
a homologically shifted analogue of the Cartier duality functor $\Hom(-,\Gm)$. 
Since $B\Gm$ is a stack and not a scheme, the 1-dual is defined on the category of commutative group stacks, the details of which will be reviewed in §\ref{subsection: cgs}.
Recall from the introduction that a pairing $\mathcal{X} \times \mathcal{Y} \to B\Gm$ of group stacks induces \emph{geometric 1-duality} if the induced maps $\cY \to \cX^D$ and $\cX \to \cY^D$ are equivalences.
Our main goal is to discuss to what extent geometric 1-duality holds for the pairs $(G, BG^\vee)$ and $(G^\vee, BG)$ when $G\to \Spec R$ is a flat Mittag-Leffler affine commutative group scheme. 
In both cases, there are examples of nonvanishing $\Extu^1(-,\Gm)$ in the fppf topology which shows that $(G,BG^\vee)$ and $(G^\vee,BG)$ need not always be geometrically dual. 
Computing the 1-dual of a group requires using a finer topology. In §\ref{subsection: 1-dual of G dual} we study the 1-dual of $G^\vee$ using the fpqc topology; in §\ref{subsection: h-torsors} we study the 1-dual of $G$ using the $h$-topology.


\subsection{Commutative group stacks and 1-duals}\label{subsection: cgs}

To discuss geometric 1-duality, we need some foundations of a theory of commutative group stacks. 
For our purposes, it suffices to work with Deligne's stricly commutative Picard stacks \cite[XVIII, §1.4]{SGA4III}. Duality of strictly commutative Picard stacks was recently studied in \cite{Bro21}, which may serve as a reference for basic properties of the duality. In modern terms, commutative group stacks may also be viewed as certain higher stacks with additional structure.

\begin{definition}[\cite{SGA4III}, XVIII, Definition 1.4.2]
    A \emph{Picard category} is a symmetric monoidal groupoid $(\mathcal{C}, +)$ such that 
    \begin{enumerate}
        \item every object $c \in \mathcal{C}$ is invertible;
        \item for all objects $c \in \mathcal{C}$, the commutativity constraint $c + c \overset{\sim}{\to} c+c$ is equal to the identity of $c+ c$.
    \end{enumerate}
\end{definition}
\begin{definition}[\cite{SGA4III}, XVIII, Definition 1.4.5]
    A \emph{commutative group stack} $\mathcal{X}$ over $\Spec R$ is a sheaf of Picard categories in the fppf topology over $\Spec R$.
\end{definition}

For example, $\cA \in \Shvab$ defines a commutative group stack with objects $\cA$ such that each object has trivial automorphism group, which we will also denote by $\cA$.
Further, the classifying stack $B\cA$, assigning $U \to \Spec R$ to the groupoid of $\cA$-torsors over $U$, is a commutative group stack.
More generally, suppose that $[A_{1} \to A_0]$ is a two-term complex in $\Shvab$.
Then there is a commutative group stack $[A_0/A_{1}]$ which is the sheafification of the assignment of groupoids $U \mapsto A_0(U)/A_1(U)$.

If $\mathcal{X}$ and $\mathcal{Y}$ are commutative group stacks, then $\Homs(\mathcal{X},\mathcal{Y})$ is the commutative group stack whose objects are strictly monoidal functors and whose morphisms are natural transformations compatible with the symmetric monoidal structure. We will need that if $[A_1 \to A_0]$ and $[B_1 \to B_0]$ are two-term complexes in $\Shvab$,
then $\Homs(A_0/A_1,B_0/B_1)$ is computed by the two-term complex 
\[\tau_{\leq 0}R\Hom([A_1 \to A_0],[B_1 \to B_0])
\]
\cite[XVIII, 1.4.18]{SGA4III}.

\begin{definition}
    If $\mathcal X$ is a commutative group stack, then the \emph{1-dual} $\mathcal{X}^D$ of $\mathcal X$ is defined by $\mathcal{X}^D = \Homs(\mathcal{X}, B\Gm)$.
\end{definition}

If $\cA \in \Shvab$, then the canonical pairing $\cA^\vee \times \cA \to \Gm$ induces pairings $\cA^\vee \times B\cA \to B\Gm$ and $\cA \times B\cA^\vee \to B\Gm$.
\begin{lemma}[\cite{Bro21}, Corollary 3.5]\label{lemma: 1-dual of classifying stack}
    If $\cA \in \Shvab$, then the pairing $\cA^\vee \times B\cA \to B\Gm$ induces an equivalence 
    \[\cA^\vee \overset{\sim}{\to} (B\cA)^D\]
\end{lemma}
\begin{proof}
    Since $B\cA$ is computed by the two-term complex $[\cA \to 0]$, $\Homs(B\cA, B\Gm)$ is computed by the two-term complex $\tau_{\leq 0}R\Hom(\cA[1],\Gm[1]) \simeq \cA^\vee$.
\end{proof}

The map $B\mathcal A^\vee \to \mathcal A^D$ induced by the pairing is more complicated:

\begin{proposition}\label{prop: 1-dual and ext1-vanishing}
    If $\cA \in \Shvab$, then the map 
    \[ B\cA^\vee \to \cA^D\]
    induced by the pairing $\cA \times B\cA^\vee \to B\Gm$ is an equivalence if and only if 
    \[ \Extu^1(\cA, \Gm) = 0,\]
    where $\Extu^1$ is the internal Ext sheaf in $\Shvab$.
\end{proposition}
\begin{proof}
    The 1-dual $\cA^D$ is computed by the complex $\tau_{\leq 0}R\Hom(\cA, \Gm[1])$. 
The homotopy sheaves of this stack are given by 
\begin{align*}
    \pi_0(\cA^D) &= \Extu^1(\cA,\Gm)\\
    \pi_1(\cA^D) &= \Homs(\cA,\Gm) = \cA^\vee.
\end{align*}
The map $B\cA^\vee \to \cA^D$ induced by the pairing $\cA \times B\cA^\vee \to B\Gm$ is the identity $\cA^\vee \to \cA^\vee$ on $\pi_1$. 
Since $\pi_0(B\cA^\vee) = 0$, the map $B\cA^\vee \to \cA^D$ is an equivalence if and only if $\Extu^1(\cA,\Gm) = 0$. 
\end{proof}

\begin{corollary}\label{corollary: when is A 1-dualizable}
    Suppose that $\cA \in \Shvab$.
    Then the pairing $\cA \times B\cA^\vee \to B\Gm$ is perfect
    if and only if $\Extu^1(\cA,\Gm) = 0$ and the canonical map $\cA \to \cA^{\vee\vee}$ is an isomorphism.
\end{corollary}
\begin{proof}
    By Proposition \ref{prop: 1-dual and ext1-vanishing}, the map $B\cA^\vee \to \cA^D$ is an equivalence if and only if $\Extu^1(\cA,\Gm)=0$.
    The pairing $\cA \times B\cA^\vee \to B\Gm$ factors as 
    \[ \cA \times B\cA^\vee \to \cA^{\vee\vee} \times B\cA^\vee \to B\Gm,\]
    where the first map is the product of the canonical $\cA \to \cA^{\vee\vee}$ and the identity, while the second map is formed by taking classifying stacks on $\cA^{\vee\vee} \times \cA^\vee \to \Gm$.
    Thus the induced map $\cA \to (B\cA^\vee)^D$ factors as 
    \begin{equation}\label{eq: induced map in perfect pairing}
        \cA \to \cA^{\vee\vee} \to (B\cA^\vee)^D.
    \end{equation}
    By Lemma \ref{lemma: 1-dual of classifying stack}, the second map in \eqref{eq: induced map in perfect pairing} is an equivalence, so $\cA \to (B\cA^\vee)^D$ is an equivalence if and only if $\cA \to \cA^{\vee\vee}$ is an isomorphism.
\end{proof}

\subsection{The 1-dual of a group ind-finite ind-scheme}\label{subsection: 1-dual of G dual}

Let us now find the 1-dual of a coflat commutative group ind-finite ind-scheme. As above $G \to \Spec R$ is a flat Mittag-Leffler affine commutative group scheme, and $G^\vee$ is our coflat commutative group ind-finite ind-scheme. 

Recall that the classifying stack $BG$ is the stack of $G$-torsors in the fppf topology. However, since $G$ is not assumed to be of finite presentation, it is natural to consider a larger class of torsors.

\begin{definition} A  \emph{$G$-torsor in fpqc topology} is a scheme $P\to\Spec R$ equipped with an action of $G$ such that the morphism $G\times_RP\to P\times_RP:(g,p)\mapsto (gp,p)$ is an isomorphism and the map $P\to\Spec R$ is surjective in the flat topology.
\end{definition}

By the same argument as Lemma \ref{lemma: extension of ML groups}, if $P$ is a $G$-torsor in fpqc topology, then $P$ is a flat Mittag-Leffler affine scheme over $\Spec R$. Denote by $B^{fpqc}G$ the moduli stack of 
$G$-torsors in fpqc topology. The sum of torsors turns $B^{fpqc}G$ into a commutative group stack.

Any $G$-torsor $P\in BG$ (that is, in the fppf topology) is automatically also a torsor in fpqc topology; thus, we get a morphism of commutative group stacks
\begin{equation}\label{eq:fppf to fpqc}
    BG\to B^{fpqc}G.
\end{equation}
The map identifies $B^{fpqc}G$ with the sheafification of $BG$ in the fpqc topology.

\begin{lemma}
The morphism \eqref{eq:fppf to fpqc} is fully faithful; if $G\to\Spec R$ is of finite presentation, then the morphism \eqref{eq:fppf to fpqc} is an equivalence. \qed
\end{lemma}

\begin{remark} Here is a larger class of affine group schemes $G$ for which the morphism
\eqref{eq:fppf to fpqc} is an isomorphism: suppose $G$ is representable as a cofiltered limit \[G=\lim G_\alpha,\]
where all $G_\alpha$ are of finite presentation and all connecting maps $G_\alpha\to G_\beta$ are surjective with kernel being a vector group.
Then any $P\in B^{fpqc}G$ is trivialized over the fppf cover 
\[P_{G_\alpha}=(G_\alpha\times P)/G\]
for any $\alpha$, therefore, $P\in BG$. 
For the authors, the main example of such $G$ is the group of positive formal loops with values in a smooth algebraic group scheme.
\end{remark}

\begin{lemma} \label{lemma: ind-triviality}
    Let $G \to \Spec R$ be a flat Mittag-Leffler affine commutative group scheme and 
    \[ 0 \to \Gm \to E \to G^\vee \to 0\]
    be an extension in $\Shvab$. 
    Then there is an fpqc cover $\Spec R' \to \Spec R$ such that the extension splits after base change to $\Spec R'$.
    If $G$ is of finite presentation over $R$, then $R'$ can be chosen to be of finite presentation over $R$.
\end{lemma}
\begin{proof}
    We will construct a retract of $\Gm \to E$ over some fpqc cover.
    The map $p:E\to G^\vee$ makes $E$ into a $\Gm$-torsor over the ind-scheme $G^\vee$. Denote by $\ell_E$ the corresponding line bundle on $G^\vee$. The pushforward $p_*\OO_E$ is a sheaf of algebras over $G^\vee$,
which can be written as
\[p_*\OO_E = \bigoplus_{n\in\ZZ}\ell_E^{\otimes n}.\]

Let us represent $G^\vee$ as filtered colimit of closed affine subschemes
\[G^\vee=\colim G^\vee_\alpha,\]
so that 
\[G^\vee=\SpInd(A)\]
for the pro-algebra
\[A=\lim A_\alpha,\qquad A_\alpha=\Gamma(G^\vee_\alpha,\OO_{G_\alpha^\vee}).\]
Put
\[M^{(n)}=\lim M^{(n)}_\alpha,\qquad M^{(n)}_\alpha=\Gamma(G^\vee_\alpha,\ell_E^{\otimes n}).\]
Since each $G_\alpha^\vee$ is affine, we have $M_\alpha^{(n)} = (M_\alpha^{(1)})^{\otimes n}$.
Therefore,
\[ E = \SpInd(\tilde A), \qquad \tilde A = \lim_\alpha\left(\bigoplus_{n \in \ZZ} M_{\alpha}^{(n)}\right).\]

From the definition, we see that for any $n$, $M^{(n)}$ is Mittag-Leffler. Moreover, we can check that $M^{(n)}$ is pro-projective as a pro-$R$-module:

By Remark \ref{remark: pro-projective conditions}, it suffices to show that for each $\alpha$, there exists $\beta \geq \alpha$ such that $\Ext^1_R(M_\alpha^{(n)}, -) \to \Ext^1_R(M_\beta^{(n)},-)$ is zero.
Pick $\beta$ so that the map $\Ext^1_R(A_\alpha,-) \to \Ext^1_R(A_\beta,-)$ is zero.
Since $\ell_E$ is locally free of rank 1, there is a 
Zariski open cover of $\Spec(A_\beta)$ on which $M_\alpha^{(n)}$ and $M_\beta^{(n)}$ are trivialized; the desired vanishing of $\Ext^1_R(M_\alpha^{(n)},- ) \to \Ext^1_R(M_\beta^{(n)},-)$ follows.

We can now apply the linear duality: the linear dual $(M^{(n)})^\vee$ is a flat 
commutative $R$-algebra.
Let $R' = (M^{(1)})^\vee$.
Then the unit $\unit_{M^{(1)}}: R \to M^{(1)} \otimes_R (M^{(1)})^\vee = M^{(1)} \otimes_R R'$ is a grouplike element of $ev(\tilde A \otimes R')$ of weight 1
and so defines a retract of 
\[\Gm\times_R\Spec(R') \to E\times_R\Spec(R').\] 

Let $P = \Spec R'$. Then $P \to \Spec R$ is an fpqc cover.
Moreover the $A$-module structure on $M^{(1)}$ makes $R'$ into an invertible $A^\vee$-comodule of rank 1. In geometric terms, $P \to \Spec R$ has the structure of a $G$-torsor in the fpqc topology.

If $G\to \Spec R$ is of finite presentation, we claim that $P\to\Spec R$ is also of finite presentation. Indeed, morphisms of finite presentation satisfy fpqc descent \cite[\href{https://stacks.math.columbia.edu/tag/00QQ}{Lemma 10.126.2}]{stacks-project} and $P \times_{\Spec R} P \to P$ is of finite presentation. 
\end{proof}

As a consequence, we obtain the vanishing of $\Extu^1(-,\Gm)$ with the dual of a finite presentation flat-Mittag-Leffler affine commutative group scheme:

\begin{corollary}\label{corollary: ext1 vanishing for dual of finite presentation G}
    Suppose that $G \to \Spec R$ is a flat Mittag-Leffler affine commutative group scheme of finite presentation. Then 
    \[ \Extu^1(G^\vee,\Gm) = 0\]
    (we emphasize that $\Extu^1$ is in the fppf topology). \qed
\end{corollary}

When $R$ is a field, Corollary \ref{corollary: ext1 vanishing for dual of finite presentation G} was previously known using the structure theorem for affine algebraic groups \cite[Lemma 1.14]{Rus13}. 
Corollary \ref{corollary: ext1 vanishing for dual of finite presentation G} also recovers that $\Ext^1(G,\Gm) = 0$ when $G \to \Spec R$ is a finite locally free group scheme \cite[Exposé VIII, Proposition 3.3.1]{SGA7I}.

As a further consequence, it follows that when $G\to \Spec R$ is of finite presentation, then the pairing $G^\vee \times BG \to B\Gm$ is perfect: 
\begin{corollary}
    If $G \to \Spec R$ is a flat Mittag-Leffler affine commutative group scheme of finite presentation, then the Cartier duality pairing $G \times G^\vee \to \Gm$ induces a perfect pairing $BG \times G^\vee \to B\Gm$: it induces equivalences 
    \begin{align*}
        G^\vee&\simeq(BG)^{D}\\
        BG&\simeq(G^\vee)^{D}.
    \end{align*} 
\end{corollary}
\begin{proof}
    Combine Corollary \ref{corollary: ext1 vanishing for dual of finite presentation G} and Corollary \ref{corollary: when is A 1-dualizable}.
\end{proof}
When $G$ is not necessarily of finite presentation, the Cartier duality pairing $G\times G^\vee\to \Gm$ induces a pairing
\begin{equation}\label{eq:1-pairing fpqc}
    B^{fpqc} G\times G^\vee\to B^{fpqc}\Gm= B\Gm.
\end{equation}
\begin{proposition}\label{prop: 1-pairing fpqc}
    For any flat Mittag-Leffler affine commutative group scheme $G \to \Spec R$, the pairing \eqref{eq:1-pairing fpqc} is perfect: it induces equivalences
\begin{align*}
    G^\vee&\simeq(B^{fpqc}G)^{D}\\
    B^{fpqc}G&\simeq(G^\vee)^{D}.
\end{align*} 
\end{proposition}
\begin{proof}
Since $B\Gm$ is a sheaf in the fpqc topology, the fpqc sheafification morphism \eqref{eq:fppf to fpqc} induces an isomorphism
\[(B^{fpqc}G)^D=\Homs(B^{fpqc}G, B\Gm)=\Homs(BG,B\Gm)=(BG)^D,\]
and the first statement follows from Lemma~\ref{lemma: 1-dual of classifying stack}.
(Of course, it can also be proved directly by working in the fpqc topology.)

For the second statement, we claim that both $B^{fpqc}G$ and $(G^\vee)^D$ are fpqc gerbes. For $B^{fpqc}G$ this is tautological. For \[(G^\vee)^D=\Homs(G^\vee,B\Gm),\] we first note that it is an fpqc stacks, because both $G^\vee$ and $B\Gm$ are fpqc stacks (for $G^\vee$ this is \cite[Chapter~2, Proposition~1.2.2]{GR17II}). Now Lemma~\ref{lemma: ind-triviality} implies it is a gerbe.

It remains to check that the morphism $B^{fpqc}G\to(G^\vee)^{D}$ induces an isomorphism on $\pi_1$. This follows from the Cartier duality between $G$ and $G^\vee$ being perfect.
\end{proof}

\begin{remark} In Proposition \ref{prop: 1-pairing fpqc}, the equivalence 
\[B^{fpqc}G\to(G^\vee)^{D}\]
factors through an equivalence of certain Hopf pro-algebras.

First, $B^{fpqc} G$ is equivalent to the moduli stack of extensions 
\[ 0 \to G\to E\to\ZZ\to 0\] 
in fpqc topology.
Letting $N^{(k)}$ be the ring of functions of $E \times_{\ZZ} k$,
we obtain a Hopf algebra in $\ProMod_R$ of the form
\begin{equation}\label{equation: hopf pro-algebra for extension by G}
    \prod_{k\in\ZZ} N^{(k)}
\end{equation}
such that each $N^{(k)}$ is a flat $R$-module.
It is commutative and cocommutative; the multiplication is compatible with the product and comultiplication satisfies
$\Delta(e_k)=\sum_{i+j=k}e_i\otimes e_j$
where $e_k\in N^{(k)}$ is the unit.
Since $G \to \Spec R$ is flat Mittag-Leffler and $E \times_{\ZZ} k$ is a $G$-homogeneous space, each $N^{(k)}$ is also flat Mittag-Leffler.

Linear duality provides an equivalence between Hopf pro-algebras of the form \eqref{equation: hopf pro-algebra for extension by G} and 
commutative cocommutative Hopf pro-algebras of the form
\begin{equation}\label{equation: hopf pro-algebra for extension of G dual}
    \bigoplus_{k\in\ZZ}M^{(k)}
\end{equation}
where each $M^{(k)}$ is a coflat pro-finite pro-$R$-coalgebra, and the comultiplication is compatible with the grading. The multiplication is graded and induces isomorphisms $M^{(i)} \otimes_{M^{(0)}} M^{(j)} \cong M^{(i+j)}$.

Hopf pro-algebras of the form \eqref{equation: hopf pro-algebra for extension of G dual} are equivalent to fppf extensions
\[ 
1 \to \Gm \to F \to G^\vee \to 0.
\]
For given such a Hopf pro-algebra, $M^{(1)}$ is an invertible $M^{(0)}$-module and thus defines a line bundle $F \to G^\vee$\footnote{Invertible objects in pro-modules are compact, and since $\Mod_{R'}$ is idempotent-complete, the subcategory of compact objects in $\ProMod_{R'}$ is exactly $\Mod_{R'}$. Thus the restriction of $M^{(1)}$ to any closed subscheme of $G^\vee$ is an invertible module.}. 
The map $F \to G^\vee$ is surjective in fppf topology since $F \to G^\vee$ is a line bundle.
The comultiplication and antipode make $F$ into a group extension of $G^\vee$ with kernel $\Gm$. 
Conversely, suppose that 
\[ 1 \to \Gm \to F \to G^\vee \to 0\]
is an extension in $\Shvab$. Then $p: F \to G^\vee$ makes $F$ into a $\Gm$-torsor over the ind-scheme $G^\vee$. Denote by $\ell_F$ the corresponding line bundle on $G^\vee$. The pushforward $p_*\OO_F$ is a sheaf of algebras over $G^\vee$,
which can be written as
\[p_*\OO_F = \bigoplus_{n\in\ZZ}\ell_F^{\otimes n}.\]
If $G^\vee = \colim_\alpha G^\vee_\alpha$, set
\[M^{(n)}_\alpha=\Gamma(G^\vee_\alpha,\ell_F^{\otimes n}), \qquad M^{(n)}=\lim M^{(n)}_\alpha.\]
The commutative group structure on $F$ gives rise to a cocommutative coalgebra structure on $\bigoplus_{n\in\ZZ} M^{(n)}$, which turns it into a Hopf algebra. Moreover, the comultiplication restricts to a coalgebra structure on each $M^{(n)}.$
From the definition, we see that for any $n$, $M^{(n)}$ is Mittag-Leffler.
Since each $G_\alpha^\vee$ is finite over $\Spec R$, each $M^{(n)}$ is pro-finite. Moreover, $M^{(n)}$ is pro-projective as a pro-$R$-module.
By Remark \ref{remark: pro-projective conditions}, it suffices to show that for each $\alpha$, there exists $\beta \geq \alpha$ such that $\Ext^1_R(M_\alpha^{(n)}, -) \to \Ext^1_R(M_\beta^{(n)},-)$ is zero.
Since $G^\vee \to \Spec R$ is coflat, there is $\beta \geq \alpha$ such that the map $\Ext^1_R(M^{(0)}_\alpha,-) \to \Ext^1_R(M^{(0)}_\beta,-)$ is zero.
Since $\ell_F$ is locally free of rank 1, there is a 
Zariski open cover of $\Spec(M^{(0}_\beta)$ on which $M_\alpha^{(n)}$ and $M_\beta^{(n)}$ are trivialized; the desired vanishing of $\Ext^1_R(M_\alpha^{(n)},- ) \to \Ext^1_R(M_\beta^{(n)},-)$ follows.
\end{remark}

The following example shows that when $G \to \Spec R$ is not necessarily of finite presentation, then the map $BG \to B^{fpqc}G$ need not be an equivalence:

\begin{example}\label{ex: ext1 with countable rank free group}
    Let $\Gamma$ be a free abelian group of infinite rank. We can consider $\Gamma$ as a group ind-finite ind-scheme. For any scheme $X$,
    \begin{equation}\label{eq:ext1}
        \Ext^1_X(\Gamma,\Gm)=\Hom_\ZZ(\Gamma,\Pic(X)).
    \end{equation}
    We claim that the fppf sheafification of this presheaf is non-zero. Thus, if $G = \Gamma^\vee$, Corollary \ref{corollary: when is A 1-dualizable} and Proposition \ref{prop: 1-pairing fpqc} show that the map $BG \to B^{fpqc}G$ of \eqref{eq:fppf to fpqc} is not an equivalence.    
    
    Let us take $X$ to be a smooth non-rational curve over $\C$. Then $\Pic(X)$ is not finitely generated, and we can choose an element of \eqref{eq:ext1} such that the image of $\Gamma$ in $\Pic(X)$ is not finitely generated. It now suffices to check the following claim:
    \begin{proposition} For $X$ as above, let $\pi:Y\to X$ be a fppf cover. Then the kernel of the map $\pi^*:\Pic(X)\to\Pic(Y)$
        is finitely generated.
    \end{proposition}
    Our argument is a simplification of a more general statement proved by J.~Starr on MathOverflow \cite{starrmo}.
    \begin{proof}
         Let $\xi\in X$ be the generic point. There exists a point $\eta\in \pi^{-1}(\xi)\subset Y$ whose residue field is a finite extension of that of $\xi$. Choose an irreducible locally closed subscheme $V\subset Y$ whose generic point is $\eta$. Shrinking $V$ if necessary, we can make it so $\pi$ gives a finite flat cover $\pi|_V:V\to U$ for an open subset $U\subset X$. Since the kernel of the restriction map $\Pic(X)\to\Pic(U)$ is
        finitely generated (by the prime divisors contained in the complement $X-U$), it suffices to check that the kernel of the pullback \[\pi^*:\Pic(U)\to\Pic(V)\] is finitely generated. However, the norm map shows that the kernel is annihilated by $\deg(\pi|_V)$, which implies the statement.
    \end{proof}
\end{example}

\subsection{The 1-dual of an affine group}
\label{subsection: h-torsors}

In this section, we study the 1-dual of $G$ when $G \to \Spec R$ is a flat Mittag-Leffler affine commutative group scheme. Once again, there may be extensions of $G$ by $\Gm$ which do not split locally in fppf topology, but split in some finer topology. In §\ref{subsection: 1-dual of G dual}, that finer topology was the fpqc topology, since $G \to \Spec R$ is an fpqc cover. In this section, we must use the $h$-topology, since $G^\vee \to \Spec R$ is an $h$-cover.

By Corollary \ref{corollary: when is A 1-dualizable}, the map $BG^\vee \to G^D$ induced by the pairing $G^\vee \times G \to \Gm$ is an equivalence if and only if $\Extu^1(G,\Gm) = 0$.
Here is an example showing that $\Extu^1(G,\Gm)$ need not be zero:
\begin{example}\label{ex: nontrivial extension of Ga by Gm}
    If $k$ is a field of characteristic zero, then $\Extu^1(\Ga,\Gm) \neq 0$ over $k$.
    The example is essentially due to Gabber \cite[Remark 2.2.16]{Ros23}. 
    Let $C = \Spec k[x,y]/(y^2 - x^3)$ be the cuspidal cubic; we will construct a point of $\Ext^1_C(\Ga,\mathbb G_m)$ that does not vanish fppf locally on $C$.\footnote{Gabber works with the tacnode $\Spec k[x,y] / (y^2 - x^4)$ instead, viewed as two copies of $\mathbb A^1$ glued along $\Spec k[\epsilon]/\epsilon^2$.} Our explanation is based on the same ideas, but is shortened by the use of results of de Jong.
    
    Consider the normalization map $\mathbb A^1 \to C$. We can view $C$ as $\mathbb A^1$ where the tangent vector is contracted to a point, that is, as the affine pushout 
    \[
\begin{tikzcd}[ampersand replacement=\&]
	{\Spec k[\epsilon]/\epsilon^2} \& {\mathbb A^1} \\
	{\Spec k} \& C
	\arrow[from=1-1, to=1-2]
	\arrow[from=1-1, to=2-1]
	\arrow[from=1-2, to=2-2]
	\arrow[from=2-1, to=2-2]
\end{tikzcd}.\]
    Define an extension $0 \to (\Gm)_C \to E \to (\Ga)_C \to 0$ as follows: take the trivial extension on $\mathbb A^1$, then glue it to itself along $\Spec k[\epsilon]/\epsilon^2$ by the automorphism
    \[ \begin{pmatrix} 1 & \phi \\ 0 & 1\end{pmatrix} : \Gm \times \mathbb G_a \to \mathbb G_m \times \mathbb G_a\]
    where $\phi \in \Hom_{k[\epsilon]/\epsilon^2}(\Ga,\mathbb G_m)$ is the homomorphism $\phi(x) = 1 + \epsilon x$.
    Taking the Cartier duals, we find an exact sequence
    \begin{equation}\label{eq: dual left exact sequence}
        0 \to (\hat{\mathbb G}_a)_C \to E^\vee \to \mathbb{Z}_C.
    \end{equation}
    Suppose towards contradiction that $(\mathbb G_m)_C \to E$ has a retract fppf locally on $C$. Then $E^\vee \to \mathbb{Z}_C$ has a section fppf locally on $C$, so $T := E^\vee \times_{{\mathbb Z}_C} 1 \to C$ has a section fppf locally on $C$. By \eqref{eq: dual left exact sequence}, $T \to C$ is a $\hat{\mathbb{G}}_a$-torsor over $C$.
    By a theorem of de Jong \cite[Remark 2.2.18]{Bhatt22}, $H^1_{fppf}(\Spec R,\hat{\mathbb{G}}_a) = 0$ for any $\mathbb{Q}$-algebra $R$.
    Thus $T\to C$ has a section, which implies that $0 \to (\hat{\mathbb{G}}_a)_C \to E^\vee \to \mathbb{Z}_C$ is a split extension, which implies that $0 \to (\Gm)_C \to E \to (\Ga)_C \to 0$ is a split extension.
    But this extension is not split since $\phi$ is a nontrivial homomorphism. We have reached the desired contradiction, so $[E]$ is a nontrivial point of $\Extu^1(\Ga,\Gm)(C)$.
    Further, \eqref{eq: dual left exact sequence} is not short exact, giving an example where Cartier duality is not exact.

\end{example}

While Example \ref{ex: nontrivial extension of Ga by Gm} shows that $\Extu^1(-,\Gm)$ is not always zero, there are some vanishing results known, which we collect below. 

\begin{proposition}[\cite{SGA7I}, Exposé VIII, Proposition 3.3.1]
    Suppose that $G \to \Spec R$ is a commutative group scheme which is either finite locally free or of multiplicative type. Then 
    \[ \Extu^1( G, \Gm) = 0.\]
\end{proposition}

\begin{proposition}[\cite{Ros23}, Proposition 2.2.17]
\label{prop: ext1 vanishing in characteristic p}
    Let $k$ be a field of characteristic $p > 0$ and $G$ be a affine commutative group scheme of finite type over $k$.
    Then 
    \[ \Extu^1(G, \Gm) = 0.\]
\end{proposition}
Proposition \ref{prop: ext1 vanishing in characteristic p} shows that the behavior of $\Ga$ in Example \ref{ex: nontrivial extension of Ga by Gm} is exclusively a characteristic zero phenomenon.
The proof of Proposition \ref{prop: ext1 vanishing in characteristic p} makes use of the structure theorems for algebraic groups.

Even though $\Extu^1(G,\Gm)$ is not always zero, every extension of $G$ by $\Gm$ splits when pulled back to a geometric point:

\begin{lemma}\label{lemma: ext-vanishing on geometric points}
    Let $k$ be an algebraically closed field and $G \to \Spec k$ an affine commutative group scheme.
    Then \[ \Ext^1_k(G,\Gm) = 0.\]
\end{lemma}
\begin{proof}
    Suppose that $0 \to \Gm \to E \to G \to 0$ is an extension in $\Shvab$.
    By Lemma \ref{lemma: extension of ML groups}, $E$ is an affine commutative group scheme.
    Since the map $\Gm \to E$ is injective, there exists a function $f \in k[E]$ of weight 1.
    There exists a finite-dimensional $E$-subrepresentation $V \subseteq k[E]$ containing $f$.
    Taking a composition series, there exists an irreducible subquotient of $V$ on which $\Gm$ acts by weight 1.
    As $k$ is algebraically closed and $E$ is commutative, this irreducible representation of $E$ is one-dimensional.
    Thus, we have constructed a character $E \to \Gm$ whose restriction to $\Gm$ is the identity, that is, a splitting of our extension.
\end{proof}

Lemma \ref{lemma: ext-vanishing on geometric points} suggests that the 1-dual of an affine group scheme $G$ classifies families of $G^\vee$-homogeneous spaces which are surjective on geometric points. This leads to the following definition:

\begin{definition}\label{def: h-torsor}
    Let $S$ be a scheme and let $\mathcal G \to S$ be a group ind-finite ind-scheme over $S$.
    A $\mathcal{G}$-torsor in $h$-topology is 
    an ind-finite ind-scheme $\mathcal{P} \to S$ 
    equipped with an action of $\mathcal G$ such that 
    \begin{enumerate}
        \item the morphism $\mathcal G \underset{S}{\times} \mathcal{P} \to \mathcal{P} \underset{S}\times \mathcal{P}$ given by $(g,p) \mapsto (gp,p)$ is an isomorphism;
        \item the morphism $\mathcal{P} \to S$ is surjective in $h$-topology.
    \end{enumerate}
\end{definition}
The notion of $h$-topology is reviewed in the Appendix, §\ref{subs: h-top}.
In the Noetherian setting, an ind-finite morphism is surjective in $h$-topology if and only if it is surjective on geometric points (Lemma \ref{lemma: ind-finite noetherian surjective on geometric pts}).

\begin{example}
    In Example \ref{ex: nontrivial extension of Ga by Gm},
    the morphism $T = E^\vee \times_{\ZZ_C}1 \to C$ over the cuspidal cubic $C$ is a $\hat{\mathbb{G}}_a$-torsor in $h$-topology, as the map $T \to C$ is surjective on geometric points. Example \ref{ex: nontrivial extension of Ga by Gm} shows that $T\to C$ is not fppf-locally trivial.
\end{example}

\begin{remark}
    Under mild assumptions, the condition that $\mathcal P \to S$ is an ind-finite ind-scheme is automatic. More precisely, 
    suppose that $S$ is Noetherian and admits a dualizing complex,
    and $\mathcal G \to S$ is a coflat group ind-finite ind-scheme.
    If $\mathcal P \to S$ is a prestack satisfying conditions (1) and (2) of Definition \ref{def: h-torsor}, then $\mathcal{P} \to S$ is ind-finite ind-schematic.
    Since $\mathcal{P} \to S$ is surjective in $h$-topology, there is an $h$-cover $Y \to S$ such that $\mathcal P|_Y$ admits a section, and thus $\mathcal{P}|_Y$ is representable by a coflat ind-finite ind-scheme over $Y$.
    By $h$-descent for $\IndCoh$ \cite[Chapter 5, Lemma 3.3.6]{GR17I}, the distribution coalgebra for $\mathcal{P}|_Y$ descends to $S$ (see Proposition \ref{prop: distributions via omega}).
\end{remark}

Let $\mathcal G$ be a commutative group ind-finite ind-scheme over $S$. Given $\mathcal G$-torsors $\mathcal{P}$
and $\mathcal{P}'$ in $h$-topology, the sum $\mathcal{P} + \mathcal{P}' := \mathcal{P} \times_{\mathcal{G}} \mathcal{P}'$ is also a $\mathcal G$-torsor in $h$-topology.
Thus $\mathcal{G}$-torsors in $h$-topology form a Picard category.

\begin{lemma}
    Suppose that $\mathcal G \to S$ is a coflat group ind-finite ind-scheme.
    Then assigning $U \to S$ to the groupoid of coflat $\mathcal G_U$-torsors in $h$-topology is a sheaf in the fppf topology.
\end{lemma}
The nontrivial part of the statement is that we need to check that being representable by a coflat ind-finite ind-scheme and surjectivity in $h$-topology both descend.
\begin{proof}
    It suffices to prove the claim when $S = \Spec R$ is affine.
    First, we show that coflat ind-finite ind-schemes satisfy fppf descent.
    From §\ref{subsection: proalgebras}, coflat ind-finite ind-schemes are of the form $\SpInd A^\vee$ where $A$ is a Mittag-Leffler coalgebra over $R$. By \cite[Lemma 10.95.1]{stacks-project}, flat Mittag-Leffler modules satisfy faithfully flat descent.

    Now suppose $R \to R'$ is a faithfully flat morphism of rings. By descent, we find that the category of coflat $\mathcal{G}$-spaces over $\Spec R$ is equivalent to the totalization of the category of coflat $\mathcal{G}$-spaces over $\Spec (R')^{\otimes \bullet}$.
    Let $\mathcal{P} \to \Spec R$ be a coflat $\mathcal{G}$-space,
    and let $\mathcal{P}' \to \Spec R'$ be its pullback along $\Spec R' \to \Spec R$.
    The condition that $\mathcal{G} \times \mathcal{P} \to \mathcal{P} \times \mathcal{P}$ is an isomorphism can be checked after base change to $\Spec R'$.
    Since fppf covers are $h$-covers, it is immediate that $\mathcal{P} \to \Spec R$ is surjective in $h$-topology if and only if $\mathcal{P}' \to \Spec R'$ is.
\end{proof}


\begin{definition}
    If $\mathcal G \to S$ is a coflat commutative group ind-finite ind-scheme, then $B^{h}\mathcal G$ is the commutative group stack of coflat $\mathcal G$-torsors in $h$-topology.
\end{definition}

\begin{theorem}
    Let $G \to S$ be an flat Mittag-Leffler affine commutative group scheme 
    over locally Noetherian $S$.
    Then 
    \[ G^D \simeq B^{h}G^\vee.\]
\end{theorem}
\begin{proof}
    First we define a map $\Theta: B^{h}G^\vee \to \Homs(G, B\Gm)$.
    Suppose that $U \to S$ is a Noetherian scheme and $\mathcal{P} \to U$ is a $G^\vee$-torsor in $h$-topology.
    Since $\mathcal{P} \times_U \mathcal{P} \cong G^\vee \times \mathcal{P}$, 
    $\mathcal{P}$ defines a cocycle $c \in G^\vee(\mathcal{P} \times \mathcal{P})$.
    Given $g \in G(U)$, applying $g$ to $c$ gives a cocyle $g(c) \in \Gm(\mathcal{P} \times \mathcal{P})$.
    Since $\mathcal{P} \to U$ is surjective in $h$-topology, Theorem \ref{theorem: h-descent} implies perfect complexes descend along $\mathcal{P} \to U$, so in particular line bundles descend along $\mathcal{P} \to U$.
    Thus $g(c)$ defines a line bundle $\Theta(\mathcal{P})(g) \in B\Gm(U)$.
    Similarly, if $\mathcal{P} \to \mathcal{P}'$ is an isomorphism of $G^\vee$-torsors over $U$, then $\Theta(\mathcal{P} \to \mathcal{P}')(g)$ is the isomorphism of line bundles induced by descent along $\mathcal{P} \to U$. 

    We claim that $\Theta$ is an equivalence.
    First, observe that the automorphism group of the trivial $G^\vee$-torsor is $G^\vee$.
    Tracing the definitions shows that $\pi_1(\Theta): \pi_1(B^{h}G^\vee) \to \pi_1(\Homs(G,B\Gm)) = G^\vee$ is an isomorphism.
    
    Now consider the map $\pi_0(\Theta): \pi_0(B^{h}G^\vee) \to \pi_0(\Homs(G, B\Gm)) = \Extu^1(G, \Gm).$
    First, we show that $\pi_0(\Theta)$ is surjective.
    If $0 \to \Gm \to E \to G \to 0$ is an extension of fppf sheaves over $U$, 
    then by Lemma \ref{lemma: extension of ML groups}, $E$ is an flat Mittag-Leffler affine commutative group scheme.
    Taking the Cartier dual gives a left-exact sequence
    \[0 \to G^\vee \to E^\vee \to \ZZ_U. \]
    Define $\mathcal{P} = E^\vee \times_{\ZZ_U} 1_U$.
    The map $G^\vee \times \mathcal{P} \to \mathcal{P} \times_U \mathcal{P}$ is an isomorphism since the sequence above is left-exact.
    We claim that $\mathcal{P} \to U$ is surjective in $h$-topology.
    Since $E^\vee \to U$ is coflat ind-finite ind-schematic, so is $\mathcal{P}$.
    First we show $\mathcal{P} \to U$ is surjective on geometric points.
    If $x: \Spec k \to U$ is a geometric point, then
    by Lemma \ref{lemma: ext-vanishing on geometric points}, $\Gm \to E_x$ admits a left inverse $f: E_x \to \Gm$, i.e. a point $f \in E^\vee(k)$ mapping to $1_x \in \ZZ_x$. Thus, $\mathcal{P} \to U$ is surjective on geometric points.
    Since $U$ is Noetherian and $\mathcal{P} \to U$ is ind-finite, by Lemma \ref{lemma: ind-finite noetherian surjective on geometric pts}, $\mathcal{P} \to U$ is surjective in $h$-topology.

    Now we show that $\pi_0(\Theta)$ is injective.
    Suppose that $\mathcal{P}$ is an $G^\vee$-torsor in $h$-topology over $U$ such that $\Theta(\mathcal{P}): G \to B\Gm$ is homotopic to zero.
    If $0 \to \Gm \to E \to G \to 0$ is the associated group extension, this means that there is a splitting $E \to \Gm$.
    Taking the Cartier dual gives a section $\ZZ_U \to E^\vee$ of $0 \to G^\vee \to E^\vee \to \ZZ_U$. As $\mathcal{P} = E^\vee \times_{\ZZ_U} 1$,
    we see that $\mathcal{P} \to U$ has a section, as desired.
\end{proof}




\subsection{Extensions of group ind-finite ind-schemes by affine group schemes}

In his paper introducing Cartier duality for affine group schemes over a field \cite{Car62}, Cartier made the following table of analogies between Cartier duality and Pontryagin duality (where we have updated the terminology):
\begin{center}
\begin{tabular}{|c|c|}
\hline 
    Affine group schemes & Compact groups \\
    Ind-finite group ind-schemes & Discrete groups \\
    Cartier duality & Pontryagin duality \\ \hline
\end{tabular}
\end{center}
Cartier continued, 
\begin{quote}
    Pour l'instant, nous ne savons pas définir les objets, analogues dans la correspondence ci-dessus, aux groupes commutatives localement compacts qui ne sont ni discrets, ni compacts. La question semble importante pour certaines applications géometriques.  \cite[89]{Car62}
\end{quote}

We still do not know a geometric description for the entire class of fppf sheaves $\mathcal A$ such that the pairing $\mathcal A \times \mathcal A^\vee \to \Gm$ is perfect. However, using the fpqc-local triviality of Lemma \ref{lemma: ind-triviality}, we can extend the Cartier duality equivalence to sheaves which are extensions of an ind-finite scheme by an affine scheme (under our usual hypotheses on the coordinate rings). This class is self-dual, and corresponds under Cartier's table of analogies to extensions of discrete groups by compact groups. 

\begin{proposition}\label{prop: duality for extension of ind-finite by affine}
    Let $G$ and $H$ be flat Mittag-Leffler affine commutative group schemes over $R$.
    \begin{enumerate}
        \item If 
        \[ 0 \to H \to E \to G^\vee \to 0\]
        is an extension of fpqc sheaves, then the sequence 
        \[
            0 \to G \to E^\vee \to H^\vee \to 0
        \]
        is exact, and 
        the natural map $E \to E^{\vee\vee}$ is an isomorphism;
        \item Cartier duality induces an equivalence of groupoids
        \[ \EXT_{fpqc}(G^\vee,H) \simeq \EXT_{fpqc}(H, G^\vee)\]
        by sending $E \mapsto E^\vee$.
    \end{enumerate}
\end{proposition}
\begin{proof}
    \begin{enumerate}
        \item Applying $R\Homs_{fpqc}(-,\Gm)$ to the extension gives a long exact sequence 
        \[
\begin{tikzcd}[ampersand replacement=\&]
	0 \& G \& {E^\vee} \& {H^\vee} \& {\Extu^1_{fpqc}(G^\vee, \Gm)} \& \cdots
	\arrow[from=1-1, to=1-2]
	\arrow[from=1-2, to=1-3]
	\arrow[from=1-3, to=1-4]
	\arrow[from=1-4, to=1-5]
	\arrow[from=1-5, to=1-6]
\end{tikzcd}
        \]
        By Lemma \ref{lemma: ind-triviality}, $\Extu^1_{fpqc}(G^\vee,\Gm) = 0$, so the sequence $0 \to G \to E^\vee \to H^\vee \to 0$ is exact, as desired.
        By Theorem \ref{theorem: perfect pairing for flat ML group}, the canonical maps $H \to H^{\vee\vee}$ and $G^\vee \to G^{\vee\vee\vee}$ are isomorphisms, so the map $E \to E^{\vee\vee}$ is also an isomorphism.
        \item Follows from (1).
    \end{enumerate}
\end{proof}

\begin{remark}
    In Proposition \ref{prop: duality for extension of ind-finite by affine}, both $\EXT_{fpqc}(G^\vee,H)$ and $\EXT_{fpqc}(H^\vee,G)$ can be identified with the groupoid of biextensions of $G^\vee \times H^\vee$ by $\Gm$.
\end{remark}

\begin{example}\label{ex: cc pairing}
    Let $E = L\Gm$ be the loop group of $\Gm$, that is, the Weil restriction of $\Gm$ along $R \to R\lseries{t}$.
    The loop group sits in a short exact sequence 
    \[ 0 \to L^+\Gm \to L \Gm \to Gr_{\Gm} \to 0\]
    where $L^+\Gm$ is the Weil restriction of $\Gm$ along $R \to R\series{t}$, and $Gr_{\Gm}$ is the affine Grassmannian of $\Gm$.
    The group $L^+\Gm$ is an extension of $\Gm$ by a tower of extensions of vector groups and thus is flat Mittag-Leffler affine, while $Gr_{\Gm}$ is ind-finite and coflat. 
    By Proposition \ref{prop: duality for extension of ind-finite by affine}, $E \to E^{\vee\vee}$ is an isomorphism; in fact $E$ is self-dual thanks to the Contou-Carrère pairing \cite{Con94}.
\end{example}

\begin{remark}
    For the geometric Langlands program, it is important to consider the Ran space version of Example \ref{ex: cc pairing}: let $\Sigma$ be a smooth curve over a base $S$, and let $\mathrm{Ran}_\Sigma$ be the Ran space of $\Sigma$, defined as the prestack $\colim_I \Sigma^I$ where the colimit runs over nonempty finite sets with surjective maps. 
    There is a group $L_{\mathrm{Ran}}\Gm$ over $\mathrm{Ran}_\Sigma$ together with a subgroup $L^+_{\mathrm{Ran}}\Gm$ and quotient $Gr_{\mathrm{Ran},\Gm}$. 
    By \cite[Lemma 5.1.2]{CH21}, $L^+_{\mathrm{Ran}}\Gm$ is flat Mittag-Leffler over $\mathrm{Ran}_\Sigma$. It is expected that $Gr_{\mathrm{Ran},\Gm}$ is coflat ind-finite over $\mathrm{Ran}_\Sigma$ and is the Cartier dual of $L^+_{\mathrm{Ran}}\Gm$. This is slightly stronger than the argument given in \cite[Theorem 5.2.1]{CH21}. Further, it is expected that $L_{\mathrm{Ran}}\Gm$ is Cartier self-dual under the Ran space version of the Contou-Carrère pairing, defined in \cite[(4.4.1)]{CH21} and \cite[§3.1]{BBE02}.
\end{remark}


\section{Categorical 1-duality}
\label{section: categorical 1-duality}

Let $G\to S$ be a flat Mittag-Leffler affine commutative group scheme.
In this section, we discuss to what extent categorical 1-duality, that is, an equivalence of derived categories of sheaves, holds for the pairs $(G, BG^\vee)$ and $(G^\vee, BG)$.
A prototype of such a statement is that if $A$ is a finite-dimensional coalgebra over a field, the category of $A$-comodules is equivalent to the category of $A^\vee$-modules\footnote{For finite-dimensional $A$, the equivalence between $A$-modules and $A^\vee$-comodules holds at the level of unbounded derived categories, see \cite[Lemma 2.4]{SST25}.}.
However, our groups $G$ and $G^\vee$ are not necessarily finite, so this argument does not necessarily apply. 

Both $G$ and $BG^\vee$ are examples of \emph{generalized 1-motives}: group stacks of the form $\tilde A/H^\vee$ where $H$ is an affine group scheme and $\tilde A$ is an extension of an abelian scheme by an affine group scheme.
Over a field of characteristic zero, Laumon introduced a Fourier-Mukai transform on bounded derived categories of generalized 1-motives of finite type \cite{Lau96}.
As we are concerned with Cartier duality for affine group schemes, we restrict our attention to 1-motives of the form $G/H^\vee$ where $G$ and $H$ are flat Mittag-Leffler affine over a general base ring $R$.
In §\ref{subsection: FM for 1-motives} we construct an equivalence of unbounded derived categories of sheaves between $G/H^\vee$ and $H/G^\vee$.
As in Laumon's work, these categories turn out to be module categories over certain associative algebras.
Even though statements can be formulated only in terms of quasi-coherent sheaves, the proofs use ind-coherent sheaves, which will be reviewed in §\ref{subsection: indcoh}.
Only for this reason, we must assume that the base ring is Noetherian admitting a dualizing complex and that $G$ and $H$ are of finite type.
For duality between $G$ and $BG^\vee$, we prove slightly stronger theorems in §\ref{subsection: FM for BG dual}.

Categorical duality for $(G^\vee, BG)$ is perhaps less familiar. Again care must be taken with normalizing derived categories correctly, and our most general result in §\ref{subsection: reps of G} constructs a duality of bounded below derived categories. These categories turn out to be comodule categories for certain coalgebras.

\begin{remark}
    We only compare categories of modules to categories of modules, and categories of comodules to categories of comodules.
\end{remark}

The main results of this section are:
\begin{itemize}
    \item Theorem \ref{maintheorem: fourier for 1-motives}, the Fourier-Mukai equivalence for 1-motives, proved in Theorem \ref{theorem: FM for 1-motives};
    \item Theorem \ref{maintheorem: fourier for BG^vee}, the Fourier-Mukai equivalence for $BG^\vee$, proved in Theorems \ref{theorem: FM for B of G dual} and \ref{theorem: FM for B of G dual is symmetric monoidal};
    \item Theorem \ref{maintheorem: fourier for BG}, the Fourier-Mukai equivalence for $BG$, proved in Theorem \ref{theorem: FM for BG}.
\end{itemize}

\subsection{Spectral prestacks, ind-schemes, and ind-coherent sheaves}\label{subsection: indcoh}

In order to use the machinery of ind-coherent sheaves, we must take a detour through derived algebraic geometry. However, all of the geometric objects involved in our main theorems will turn out to be classical.

Let $\DCAlg$ be the $\infty$-category of connective $\Einfty$-rings. A \emph{spectral prestack} is by definition a functor $\mathcal X : \DCAlg \to \Spc$. 
A \emph{spectral scheme} is a spectral prestack admitting a Zariski open cover by representable functors. Hereafter we will write ``scheme'' for spectral scheme and ``prestack'' for spectral prestack.

If $S = \Spec A$ is an affine scheme, then for each $n \geq 0$ there is the truncation $\t{n}S = \Spec \t{n}A$.
A prestack $\mathcal X$ is said to be \emph{convergent} if the natural map $\mathcal{X}(S) \to \lim_n \mathcal{X}(\t{n}S)$ is an equivalence for all affine schemes $S$ \cite[Chapter 1, §0.1.1]{GR17II}.

\begin{definition}
    A prestack is an \emph{ind-affine ind-scheme} if 
    \begin{itemize}
        \item $\mathcal X$ is convergent, and
        \item $\mathcal X$ is a filtered colimit of affine schemes under closed embeddings.
    \end{itemize}
\end{definition}
For $R \in \DCAlg$, let $\IndSch_R^{indaff}$ be the $\infty$-category of ind-affine ind-schemes over $\Spec R$.
The Yoneda embedding $\DCAlg_R^{op} \to \PreStk$ extends to a fully faithful functor $\Pro(\DCAlg_R)^{op} \to \PreStk$.
Since every ind-affine ind-scheme is a filtered colimit of affine schemes, ind-affine ind-schemes form a full subcategory of $\Pro(\DCAlg)^{op}$.
The next proposition identifies that full subcategory.

\begin{proposition}\label{prop: derived ind-scheme}
    Suppose that $A \in \Pro(\DCAlg_R)$ is a pro-$R$-algebra, and let $\mathcal{X} = \Hom(A,-) \in \PreStk_R$ be the associated spectral prestack.
    Then $\mathcal{X}$ is an ind-affine ind-scheme if and only if 
    \begin{itemize}
        \item the prestack $\mathcal{X}$ is convergent, and
        \item the pro-module $\pi_0(A)$ satisfies the Mittag-Leffler condition.
    \end{itemize}
\end{proposition}
\begin{proof}
    If $\mathcal X$ is an ind-scheme, then $\mathcal X$ is convergent by definition.
    Also, $\mathcal{X}$ is a filtered colimit of affine schemes along closed immersions, so $A$ is equivalent to a pro-system with surjective transition maps on $\pi_0$. Thus $\pi_0(A)$ satisfies the Mittag-Leffler condition.

    Now for the converse. If $\pi_0(A)$ satisfies the Mittag-Leffler condition, then by \ref{theorem: algebras in pro}, it is equivalent to a cofiltered limit of commutative algebras with surjective transition maps.
    Thus the classical truncation $\mathcal{X}^{cl}$ is an ind-affine ind-scheme.
    As $\mathcal X$ is convergent and $\mathcal X$ is a filtered colimit of affine schemes, $\mathcal X$ admits connective deformation theory. By \cite[Chapter 2, Corolary 1.3.13]{GR17II}, $\mathcal X$ is an ind-scheme.
    Since $\mathcal X^{cl}$ is ind-affine, so is $\mathcal X$ \cite[Chapter 2, Definition 1.6.3]{GR17II}.
\end{proof}

Suppose now that $R$ is classical, and consider let $\IndSch_R^{indf,fl}$ be the category of classical ind-finite ind-schemes that are coflat over $R$. Every $X\in\IndSch^{indf,fl}_R$ can be written as $\SpInd(A_X)$ for $A_X\in\Pro(\CAlg^f_R)^{\ML}$. If in addition $X$ is coflat, we can use the linear duality and associate to $X$ a commutative coalgebra $A_X^\vee$ that is flat and Mittag-Leffler. We call $A_X^\vee$ the \emph{distribution coalgebra of $X$}.

\begin{proposition}\label{prop: base change is classical}
    Let $R \to R'$ be a morphism of classical rings.
    Let $X \to \Spec R$ be a coflat ind-finite ind-scheme, and let 
    \[ X' = X \underset{\Spec R}{\times} \Spec R' \]
    be the derived base change of $X$ along $R \to R'$.
    Then $X' \to \Spec R'$ is a classical coflat ind-finite ind-scheme over $S$. Moreover, $A_{X'} = \OO_X \otimes_R R'$ and $A_X^\vee = A_X^\vee \otimes_R R'$.
\end{proposition} 
The proof of Proposition \ref{prop: base change is classical} is based on the following result:
\begin{proposition}[\cite{AS25}, Proposition 2.16]
\label{prop: proalg to algpro is fully faithful}
    Let $0 \leq n < \infty$ and let $\mathcal C$ be a symmetric monoidal $(n,1)$-category.
    Then the functor $\Pro(\Alg_{\Einfty}(\mathcal C)) \to \Alg_{\Einfty}(\Pro(\mathcal C))$ obtained by pro-extension of the canonical functor $\Alg_{\Einfty}(\mathcal C) \to \Alg_{\Einfty}(\Pro(\mathcal C))$ is fully faithful.
\end{proposition}

\begin{proof}[Proof of Proposition \ref{prop: base change is classical}]
    As the class of convergent prestacks is closed under limits, $X'$ is convergent. 
    So, to show that $X'$ is classical, it suffices to show $\t{n}X'$ is classical for all $n$.
    Since $X$ is a coflat ind-finite ind-scheme, $X = \Hom_R(A_X,-)$ for some classical pro-algebra $A_X$.
    By Proposition \ref{prop: proalg to algpro is fully faithful}, the hom space $\Hom_R(A_X,-)$ may be computed in $\Einfty$-algebras in pro-$R$-modules.
    Thus, the base change $\t{n}X'$ is equivalent to $\Hom_{R'}(A_X', -)$ where $A_X' = A_X \overset{L}{\otimes}_R R'$ is considered as an $\Einfty$-algebra in pro-$R'$-modules.
    Since the underlying pro-module of $A_X$ is equivalent to a Mittag-Leffler system of finite free $R$-modules, $A_X \overset{L}{\otimes}_R R'$ is a discrete pro-$R'$-module, so $\t{n}X'$ is classical.

    Theorem \ref{theorem: algebras in pro} implies that $A_{X}'$ is equivalent as a pro-algebra to a pro-system of classical finitely generated $R'$-modules with surjective transition maps.
    Hence $X'$ is a classical coflat ind-finite ind-scheme,
    and the coordinate and distribution algebras of $X'$ are the base change along $R \to R'$ of those of $X$.
\end{proof}


Now we review the relevant categories of sheaves. Given a prestack $\mathcal X$, the $\infty$-category of quasi-coherent sheaves on $\mathcal X$ is defined to be
\[ \QCoh(\mathcal X) = \lim_{\Spec B \to \mathcal X} D(B)\]
where $D(B)$ is the $\infty$-category of $B$-modules in spectra \cite[Chapter 3, §1]{GR17I}.
When dealing with sheaves on ind-schemes, we will also need the notion of ind-coherent sheaves, as developed by Gaitsgory and Rozenblyum \cite{GR17I,GR17II}, which has better properties with respect to base change.

\begin{remark}
    The text \cite{GR17I,GR17II} defines $\IndCoh$ for prestacks locally almost of finite type (laft) over a field of characteristic zero. 
    However, the properties of $\IndCoh$ we need also hold with the same definitions for laft prestacks over any Noetherian base ring $R$ admitting a dualizing complex.
\end{remark}

Let $R$ be a Noetherian commutative ring.
For an $R$-scheme $X$ locally almost of finite type, by definition $\IndCoh(X)$ is the ind-completion of the full subcategory of $\QCoh(X)$ of bounded complexes with coherent cohomology.
The assignment $X \rightsquigarrow \IndCoh(X)$ is functorial in $A$ with respect to the shriek-pullback: given $f: X \to Y$, we have 
\[ f^!: \IndCoh(Y) \to \IndCoh(X)\]
which is left adjoint to $f_*^{\IndCoh}$ when $f$ is an open immersion, and right adjoint to $f_*^{\IndCoh}$ when $f$ is proper.
If $\mathcal X$ is a laft prestack over $R$, then the $\infty$-category of ind-coherent sheaves on $\mathcal X$ is by definition
\[ \IndCoh(\mathcal X) = \lim_{\Spec A \to \mathcal X \text{ laft}} \IndCoh(\Spec A).\]
Now assume that $R$ admits a dualizing complex $\omega_R$.
The dualizing object for $\pi: \mathcal X \to \Spec R$ is by definition $\omega_X = \pi^! \omega_R$. The dualizing object $\omega_{\mathcal X}$ is the tensor unit of $\IndCoh(\mathcal X)$.
If $f: \mathcal X \to \mathcal Y$ is ind-proper, then $f^!$ has left adjoint $f_*^{\IndCoh}$, for which we will also write $f_!$.

If $\pi: \mathcal X \to S= \Spec R$ is a coflat ind-finite ind-scheme, there is another description of the distribution coalgebra arising from functoriality of $\IndCoh(-)$. 
For any $\mathcal X \to S$ locally almost of finite type, we have the functor $\Upsilon_{\mathcal X}: \QCoh(\mathcal X) \to \IndCoh(\mathcal X)$ given by tensoring with $\omega_{\mathcal X}$.
If $\mathcal X$ is eventually coconnective, then $\Upsilon_{\mathcal X}$ is fully faithful \cite[Lemma 10.3.4]{Gai13_indcoh}.
Since $\pi^{\IndCoh}_\ast$ is left adjoint to the monoidal functor $\pi^!$, $\pi^{\IndCoh}_\ast$ is oplax monoidal, so $\pi^{\IndCoh}_{\ast}(\omega_{\mathcal X})$ naturally has the structure of a coalgebra.

\begin{proposition} \label{prop: distributions via omega}
Let $R$ be a Noetherian ring admitting a dualizing complex and let $\pi:X\to S=\Spec R$ be a coflat ind-finite ind-scheme.
Then 
$\pi_*^{\IndCoh}(\omega_X)=\Upsilon(A_X^\vee)$, where $A^\vee\in\QCoh(S)^\heartsuit$ is the distribution coalgebra of $\pi:X\to S$. 
\end{proposition}
\begin{proof} 
    The proof reduces to the case when $X \to S$ is finite, in which case it follows from Serre duality.
\end{proof}

The assignment $X \rightsquigarrow \pi_!\omega_X$ is functorial in the following way:
suppose that $f: X \to Y$ is a morphism in $\IndSch_R^{indf,cofl}$.
As $f^! \omega_Y = \omega_X$,
we have a natural morphism
$ (\pi_X)_!\omega_X \simeq (\pi_Y)_! f_! f^!\omega_Y \to (\pi_Y)_!\omega_Y$,
which is thus a morphism $\Upsilon(A_X^\vee) \to \Upsilon(A_Y^\vee)$.
\begin{lemma}\label{lemma: functoriality of distribution coalgebra}
    The map $\Upsilon(A_X^\vee) \to \Upsilon(A_Y^\vee)$ is linearly dual to $f^*: A_Y \to A_X$.
\end{lemma}
\begin{proof}
    Suppose $f: X \to Y$ is a finite morphism, so that $f^! = \Hom_{\OO_Y}(\OO_X,-)$. Our morphism $(\pi_X)_! \omega_X \to (\pi_Y)_!\omega_Y$ is identified with 
    \[ \Hom_{(\pi_Y)_*\OO_Y}((\pi_X)_*\OO_X, \Hom_{\OO_S}((\pi_Y)_*\OO_Y, \omega_S)) \to \Hom_{\OO_S}((\pi_Y)_*\OO_Y, \omega_S),\]
    induced by $f^*: (\pi_Y)_*\OO_Y \to (\pi_X)_*\OO_X$.
    Under hom-tensor adjunction, this corresponds to $((\pi_X)_*\OO_X)^\vee \otimes \omega_S \to ((\pi_Y)_*\OO_Y)^\vee \otimes \omega_S$
    dual to $f^*$. Taking the colimit over $X$ gives the result for ind-finite $X$.
\end{proof}

\subsection{Fourier-Mukai transforms for affine 1-motives}
\label{subsection: FM for 1-motives}

Throughout this section, $R$ is a classical Noetherian ring admitting a dualizing complex and $S = \Spec R$.

\begin{definition}\label{definition: laumon groupoid}
    A \emph{Laumon groupoid} over $R$ is a groupoid 
    $X_1\rightrightarrows X_0\to S$ such that 
    \begin{enumerate}
        \item $X_0\to S$ is flat and affine of finite type;
        \item Both maps $X_1\to X_0$ are ind-finite ind-schematic and coflat. 
    \end{enumerate}
\end{definition}

\begin{example}[Affine 1-motives]
    Suppose that $G$ and $H$ are flat Mittag-Leffler affine commutative group schemes of finite type, and $H^\vee \to G$ is a homomorphism.
    Then the action groupoid $H^\vee \times G \rightrightarrows G$ is a Laumon groupoid with quotient stack $[G/H^\vee]$.
    We will call such a quotient stack an \emph{affine 1-motive}.
\end{example}

\begin{example}[de Rham groupoid]
    Assume $R$ is a $\mathbb{Q}$-algebra.
    Suppose that $X \to S$ is a smooth affine morphism.
    Let $(X \times X)^{\wedge}_X$ be the completion of $X \times X$ along the diagonal.
    Then both the projections $(X \times X)^{\wedge}_X \rightrightarrows X$ are ind-finite ind-schematic. Since the jet bundles of $X \to S$ are locally free, the projections are also coflat. 
    Thus $(X \times X)^{\wedge}_X \rightrightarrows X$ is a Laumon groupoid, with quotient stack the \emph{de Rham stack} $X_{dR}$ of $X$.
\end{example}

\begin{remark}
As in Remark \ref{remark: derived Cartier duality?}, it is natural to wonder whether Definition \ref{definition: laumon groupoid} has an $\infty$-categorical counterpart. 
    Definition \ref{definition: laumon groupoid} is sufficient for 1-categorical work but not higher (for example, the de Rham stack of a singular morphism is not of this form).
    The advantage of Definition \ref{definition: laumon groupoid} is that the resulting associative algebras are classical.
    We'll wait for others to develop the derived version.
\end{remark}

If $G$ and $H$ are flat Mittag-Leffler affine commutative group schemes over $S$, then by Theorem \ref{theorem: perfect pairing for flat ML group}, a homomorphism $H^\vee \to G$ is equivalent to a bilinear pairing $G^\vee \times H^\vee \to \Gm$, which in turn is equivalent to a homomorphism $G^\vee \to H$.
Thus, if $G^\vee \times H^\vee \to \Gm$ is a bilinear pairing, then there are two affine 1-motives $\mathcal X = [G/H^\vee]$ and $\mathcal Y= [H/G^\vee]$ arising as quotients of Laumon groupoids. Our goal is to construct a Fourier-Mukai transform relating their derived categories of sheaves.

Given a Laumon groupoid $X_1 \rightrightarrows X_0$ over $R$,
consider the quotient stack $\cX:=[X_0/X_1]$, and consider the natural projection
\[q:X_0\to\cX.\]
\begin{proposition} \label{prop: shriek pullback is monadic}
    The functor $q^!:\IndCoh(\cX)\to\IndCoh(X_0)$ is monadic.
\end{proposition}
\begin{proof} 
    Since $\IndCoh^!$ is right Kan-extended from affines, it takes colimits to limits, so $\IndCoh(\cX) \simeq \Tot \IndCoh(X_\bullet)$ where $X_\bullet = (X_1)^{\times \bullet}_{X_0}$.
    Under this identification, $q^!: \IndCoh(\cX) \to \IndCoh(X_0)$ is equivalent to the tautological map $\Tot \IndCoh(X_\bullet) \to \IndCoh(X_0)$.
    As the maps $X_1 \rightrightarrows X_0$ are ind-schematic and ind-proper, \cite[Chapter 3, Proposition 3.3.3(a)]{GR17II} shows that $\IndCoh(\cX) \to \IndCoh(X_0)$ is monadic.
\end{proof}

Denote the corresponding monad by $P$; by base change, the underlying functor $\IndCoh(X_0)\to\IndCoh(X_0)$
is given by $t_*^{\IndCoh}s^!$ for the structure maps $s,t:X_1\rightrightarrows X_0$.

\begin{proposition} \label{prop:monad on QCoh}
The monad $P$ preserves $\Upsilon_{X_0}(\QCoh(X_0))\subset\IndCoh(X_0)$.
\end{proposition}
\begin{proof}
Since $X_0$ is affine, $\Upsilon(\QCoh(X_0))$ is generated by $\omega_{X_0}$, so it suffices to verify that
\[t_*^{\IndCoh}s^!(\omega_{X_0})\in\Upsilon(\QCoh(X_0)).\]
This follows from Proposition~\ref{prop: distributions via omega}.
\end{proof}
By Proposition \ref{prop:monad on QCoh}, $q^!$ restricts to a monadic functor from the full subcategory spanned by objects $E \in \IndCoh(\mathcal X)$ such that $q^!E \in \Upsilon_{\mathcal X_0}(\QCoh_{X_0})$. This turns out to be exactly $\QCoh(\mathcal X)$:

\begin{proposition}\label{prop: pullback detects QCoh}
    For $E \in \IndCoh(\mathcal X)$, $E$ is in the essential image of $\Upsilon_{\mathcal X}: \QCoh(\mathcal X) \to \IndCoh(\mathcal X)$ if and only if $q^!E$ is in the essential image of $\Upsilon_{X_0}: \QCoh(X_0) \to \IndCoh(X_0)$.
\end{proposition}
\begin{proof}
    Again let $X_\bullet = (X_1)^{\times \bullet}_{X_0}$.
    Since $\IndCoh^!$ and $\QCoh^*$ are right Kan extended from affine schemes, they take colimits to limits, so 
    $\IndCoh(\mathcal{X}) \simeq \Tot \IndCoh(X_\bullet)$ and $\QCoh(\mathcal X) \simeq \Tot \QCoh(X_\bullet)$.
    Consider the commutative diagram 
    \[
\begin{tikzcd}[ampersand replacement=\&]
	{\QCoh(\mathcal{X})} \& {\Tot \QCoh(X_\bullet)} \& {\QCoh(X_0)} \\
	{\IndCoh(\mathcal{X})} \& {\Tot \IndCoh(X_\bullet)} \& {\IndCoh(X_0)}
	\arrow["\sim", from=1-1, to=1-2]
	\arrow["{\Upsilon_{\mathcal X}}", from=1-1, to=2-1]
	\arrow[from=1-2, to=1-3]
	\arrow["{\Upsilon_{X_\bullet}}", from=1-2, to=2-2]
	\arrow["{\Upsilon_{X_0}}", from=1-3, to=2-3]
	\arrow["\sim", from=2-1, to=2-2]
	\arrow[from=2-2, to=2-3]
\end{tikzcd}
    .\]
    The composite of the rows is the pullback with respect to $q$.
    If $E \in \IndCoh(\mathcal X)$ and $q^!E $ is in the essential image of $\Upsilon_{X_0}$, then the pullback of $E$ to each $X_\bullet$ is in the image of $\Upsilon_{X_\bullet}$.
    By Proposition \ref{prop: base change is classical}, each $X_n$ is classical, so $\Upsilon_{X_n}$ is fully faithful for all $n$.
    Thus $E$ is equivalent to an object in the image of $\Tot \Upsilon_{X_\bullet}$, as desired.
\end{proof}


\begin{theorem} \label{theorem: Laumon groupoid monad is algebra}
Let $X_1 \rightrightarrows X_0$ be a Laumon groupoid over $S$ and let $q: X_0 \to \mathcal X = [X_0/X_1]$ be the quotient map. 
\begin{enumerate}
    \item The functor $q^*: \QCoh(\cX) \to \QCoh(X_0)$ is monadic. The associated monad $P^{\QCoh}$ is the restriction of $P$ along $\Upsilon_{X_0}: \QCoh(X_0) \to \IndCoh(X_0)$.
    \item 
There exists an associative $S$-algebra $A_1$, together with a morphism $A_0\to A_1$, where $X_0=\Spec(A_0)$, such that $P^{\QCoh}$ is the monad $A_1\otimes_{A_0}-$. Moreover, $A_1$ as a right (resp. left) $A_0$-module is identified with the distribution algebra of $s:X_1\to X_0$ (resp. $t:X_1\to X_0$). In particular, $A_1$ is flat as a right and as a left $A_0$-module.
    \item $\QCoh(\cX)$ is equivalent to the category of $A_1$-modules in $\QCoh(S)$.
\end{enumerate}
\end{theorem}
\begin{proof} 
    \begin{enumerate}
        \item By Propositions \ref{prop: shriek pullback is monadic} and \ref{prop:monad on QCoh},
        $q^!$ restricts to a monadic functor on the full subcategory of $\IndCoh(\cX)$ spanned by those $E$ such that $q^!E \in \Upsilon_{X_0}(\QCoh(X_0))$. By Proposition \ref{prop: pullback detects QCoh}, that category is exactly $\Upsilon_{X_0}(\QCoh(X_0))$.
        As $q^!\Upsilon_{\cX} \simeq \Upsilon_{X_0} q^*$, the claim follows.
        \item It suffices to notice that the kernel of the functor underlying the monad $P^{\QCoh}$ lies in $\QCoh(X_0\times_S X_0)^\heartsuit$.
        \item  Let $\pi: X_0 \to S$ be the structure map.
        Both of the functors $q^*: \QCoh(\cX) \to \QCoh(X_0)$ and $\pi_*: \QCoh(X_0) \to \QCoh(S)$ are monadic and continuous.  
        Thus, the composition is conservative, admits a left adjoint, and is continuous. Thus, the composite $\pi_*q^*$ is monadic.
        The underlying endofunctor of the composite monad is given by $A_1 \otimes_{A_0} (A_0 \otimes -) \simeq A_1 \otimes -$. \qedhere
    \end{enumerate}
\end{proof}

Now suppose we are given flat Mittag-Leffler affine commutative group schemes $G=\Spec A$ and $H=\Spec B$ over $S$ and a bilinear pairing $G^\vee \times H^\vee \to \Gm$.
By Proposition \ref{prop: mapping flat into coflat and bilinear}, a map $G^\vee \times H^\vee \to \Gm$ is encoded by $u: A \otimes_R B \to R$.
The pairing is bilinear if and only if $u$ is a Hopf pairing, i.e. the following diagrams commute:
\[
\begin{tikzcd}[ampersand replacement=\&]
	{A \otimes A \otimes B} \& {A \otimes B} \& {A \otimes B \otimes B} \& {A \otimes B} \\
	{A \otimes A \otimes B \otimes B} \& R \& {A \otimes A \otimes B \otimes B} \& R \\
	A \&\& B \\
	{A \otimes B} \& R \& {A \otimes B} \& R
	\arrow["{\nabla_A \otimes 1}", from=1-1, to=1-2]
	\arrow["{1 \otimes \Delta_B}"', from=1-1, to=2-1]
	\arrow["u", from=1-2, to=2-2]
	\arrow["{1 \otimes \nabla_B}", from=1-3, to=1-4]
	\arrow["{\Delta_A \otimes 1}"', from=1-3, to=2-3]
	\arrow["u", from=1-4, to=2-4]
	\arrow["{u \otimes u}"', from=2-1, to=2-2]
	\arrow["{u \otimes u}"', from=2-3, to=2-4]
	\arrow["{1 \otimes \eta_B}"', from=3-1, to=4-1]
	\arrow["{\epsilon_A}", from=3-1, to=4-2]
	\arrow["{\eta_A \otimes 1}"', from=3-3, to=4-3]
	\arrow["{\epsilon_B}", from=3-3, to=4-4]
	\arrow["u"', from=4-1, to=4-2]
	\arrow["u"', from=4-3, to=4-4]
\end{tikzcd}
\]
Given the data of the Hopf pairing $u$, there is an explicit description of the algebra $A_1$ associated to $[G/H^\vee]$.

\begin{definition}
    Let $A$ and $B$ be Hopf algebras equipped with a Hopf pairing $u: A \otimes_R B \to R$.  
    The \emph{smash product} $A \sharp_u B$ is defined to be the associative $R$-algebra with underlying module $A \otimes_R B$ and product 
    \[ (a \sharp b) (a' \sharp b') = \sum u(a_{(1)}' \otimes b_{(1)})aa_{(2)}' \sharp b_{(2)} b',\]
    where we have used Sweedler notation $\Delta(-) =\sum (-)_{(1)} \otimes (-)_{(2)}$.
\end{definition}
This definition is a special case of the smash product of a Hopf algebra on its module algebra \cite[155]{Swe69}. When $H^\vee$ is a formal Lie group over a field of characteristic zero, Laumon introduced the same associative algebra in \cite[§6.1]{Lau96}.

\begin{example}
    The Cartier dual of the additive group $\Ga^\vee$ is the
    group ind-finite ind-scheme $\Ga^{\hat{\sharp}}$ formed by the PD nilpotent envelope of $0 \in \Ga$.
    A point of $\Ga^{\hat{\sharp}}$ is a sequence of divided powers $a^{(n)}$ satisfying $a^{(n)}a^{(m)} = \binom{n+m}{n} a^{(n+m)}$ and $a^{(n)} = 0$ for $n \gg 0$. In characteristic zero, this is just the formal additive group.
    Let $u: \Ga^\vee \times \Ga^\vee \to \Gm$ be the exponential pairing 
    \[ u((a^{(n)}), (b^{(n)})) = \sum_{n=0}^\infty n! a^{(n)} b^{(n)}.\]
    The smash product $\OO_{\Ga} \sharp_u \OO_{\Ga}$ has underlying vector space $R[x] \otimes_R R[y]$.
    Viewed as a function on $R[x]\otimes R[y]$, $u$ is given by $\sum_{n=0}^\infty n! \partial_x^{(n)} \otimes \partial_y^{(n)}$
    where $\partial_x^{(n)}$ and $\partial_y^{(n)}$ are dual bases to powers of $x$ and $y$.
    The multiplication on the smash product is given by the relation 
    \[ yx = u(x,y) 1 + u(x,1)y  +u(1,y)x + u(1,1)x\otimes y = xy + 1,\]
    recovering the familiar Weyl algebra relation $yx = xy+1$. 
\end{example}

\begin{proposition} \label{prop: quotient stack and smash product}
    Let $R$ be a Noetherian ring admitting a dualizing complex.
    Suppose $G=\Spec A$ and $H= \Spec B$ are flat Mittag-Leffler affine commutative group schemes of finite type over $R$.
    Suppose $u: G^\vee \times H^\vee \to \Gm$ is a bilinear pairing.
    Then $\QCoh([H/G^\vee])$ is equivalent to the category of $A\sharp_u B$-modules in $\QCoh(S)$.
\end{proposition}
\begin{proof}
    The groupoid associated to $G/H^\vee$ is $H^\vee \times G \rightrightarrows G$, where $s$ is projection onto the second factor and $t$ is the homomorphism $H^\vee \to G$, followed by multiplication on $G$.
    Let $A_0 \to A_1$ be as in Theorem \ref{theorem: Laumon groupoid monad is algebra}. 
    By definition, $A_0 = B$. We will show that $A_1$ is the smash product $A \sharp_u B$.

    The module of distributions of $s$ is evidently $A_1 = A \otimes B$ as a right $B$-module,
    and the inclusions $B \to A_1$ and $A \to A_1$ are ring homomorphisms.
    The product is determined by the left $B$-module structure on $A \otimes B$ induced by $t: H^\vee \times G \to G$. 
    On (pro-)rings of functions, $t$ is given by $(u^\vee \otimes 1)\Delta_G: B \to A^\vee \otimes B$, where $u^\vee: B \to A^\vee$
    is adjoint to $u: B \otimes A \to R$. Thus the left pro-module structure on $A^\vee \otimes B$ is given by
    $(u^\vee \otimes 1)\Delta_G$, followed by multiplication in $A^\vee \otimes B$.
    In formulas, for $b, b' \in B$ and $\xi \in A^\vee$, we have
    \[ b \cdot (\xi \otimes b') = \sum \Delta_H^\vee(u^\vee(b_{(1)})\otimes \xi) \otimes (b_{(2)}b').\]
    Taking duals with respect to the right $B$-action,
    the left $B$-module structure on $A \otimes B$ is given for $b,b' \in B$ and $a' \in A$ by
    \[ b \cdot (a' \otimes b') = \sum u(a'_{(1)}\otimes b_{(1)})a'_{(2)} \otimes b_{(2)}b',\]
    as desired.
\end{proof}

\begin{theorem}\label{theorem: FM for 1-motives}
    Let $R$ be a Noetherian ring admitting a dualizing complex.
    Let $G=\Spec A$ and $H=\Spec B$ be flat Mittag-Leffler affine commutative group schemes of finite type over $R$ and let $u: G^\vee \times H^\vee \to \Gm$ be a bilinear pairing.
    Then there is an equivalence of categories
    \[ \QCoh([G/H^\vee]) \simeq \QCoh([H/G^\vee]).\]
\end{theorem}
\begin{proof}
    By Proposition \ref{prop: quotient stack and smash product}, $\QCoh([G/H^\vee])$ and $\QCoh([H/G^\vee])$ are equivalent to the categories of $B \sharp_u A$-modules and $A \sharp_u B$-modules in $\QCoh(S)$, respectively.
    Thus, it suffices to show these two algebras are isomorphic.

    Let $\iota$ be the antipode of $B$.
    Consider $\phi: A \sharp_u B \to B \sharp_u A$ 
    defined by $\phi(a \sharp b) = a \iota(b)$.
    We will show $\phi$ is a composition of two $R$-algebra antiisomorphisms and thus is an $R$-algebra isomorphism.     
    First consider $\psi_1: A\sharp_u B \to B \sharp_u A$
    defined by $\psi_1(a \sharp b) = b \sharp a$. The map $\psi_1$ is plainly an isomorphism of $R$-modules; it is an antiautomorphism since 
    \begin{align*}
        \psi_1( (a \sharp b)(a' \sharp b') ) &= \sum u(a'_{(1)}\otimes b_{(1)}) b_{(2)}b' \sharp aa'_{(2)}  \\
        &= \sum u(a'_{(1)} \otimes b_{(1)}) b' b_{(2)} \sharp a'_{(2)} a \\
        &= (b' \sharp a')(b \sharp a) = \psi_1(a' \sharp b')\psi_1(a \sharp b),
    \end{align*}
    where we have used that $A$ and $B$ are commutative.
    Now consider $\psi_2: B \sharp_u A \to B \sharp A$
    defined by $\psi_2(b \sharp a) = a\iota(b)$.
    Since $\psi_2$ is linear with respect to the left action of $B$ and the right action of $A$,
    to check that $\psi_2$ is an antihomomorphism, it suffices to check whether $\psi_2(ab) = \psi_2(b)\psi_2(a)$ when $a \in A$ and $b \in B$.
    Now 
    \begin{align*}
        \psi_2(ab) &= \sum u(a_{(1)}\otimes b_{(1)}) a_{(2)}\iota b_{(2)}  \\
            &= \sum u(a_{(1)} \otimes b_{(1)}) u(a_{(2)} \otimes \iota b_{(2)}) \iota b_{(3)} \sharp a_{(3)}.
    \end{align*}
    Since $u$ is induced by a bilinear pairing $G^\vee \times H^\vee \to \Gm$,
    $u(x_{(1)}\otimes y)u(x_{(2)}\otimes z) = u(x\otimes yz)$.
    But $\sum b_{(1)}\iota b_{(2)} = \eta \epsilon(b)$ (the antipode axiom of a Hopf algebra), so
    \begin{align*}
        \psi_2(ab) &= \sum u(a_{(1)} \otimes \eta \epsilon(b_{(1)}) \iota b_{(2)}\sharp a_{(2)} \\
        &= \sum \epsilon(a_{(1)}) \epsilon(b_{(1)}) \iota b_{(2)} \sharp a_{(2)} \\
        &= \iota b \sharp a = \psi_2(b)\psi_2(a).
    \end{align*}
    Also, $\psi_2^2(a) = a$ for $a \in A$ and $\psi_2^2(b) = \iota ^2(b) = b$ for $b \in B$, so since $\psi_2$ is an antihomomorphism, $\psi_2$ is an involution. 

    Thus $\phi = \psi_2 \psi_1: a \sharp b \mapsto a\iota (b) = \sum u(a_{(1)} \otimes \iota b_{(1)}) \iota b_{(2)} \sharp a_{(2)}$ is an $R$-algebra isomorphism, as desired.
\end{proof}

\subsection{Representations of an ind-finite group}\label{subsection: FM for BG dual}

Theorem \ref{theorem: FM for 1-motives} gives an equivalence $\QCoh([G/H^\vee]) \simeq \QCoh([H/G^\vee])$ when $G$ and $H$ are flat Mittag-Leffler of finite type.
In the special case that $H$ is trivial, 
it is possible to write the equivalence as an integral transform only using $*$-pullback and $*$-pushforward; further, the hypothesis that $G$ is of finite type over the base may be removed.

Given a flat Mittag-Leffler affine commutative group scheme $G$, the canonical pairing $can: G \times G^\vee \to \mathbb G_m$ induces a morphism $G \times BG^\vee \to B\mathbb G_m$, which is in particular a line bundle $\mathcal L_G$ on $G \times BG^\vee$.
This line bundle can be understood via descent along $G \to G \times BG^\vee$ as the trivial line bundle on $G$ equipped with the descent datum $can$.
Our Fourier-Mukai transform will be the integral transform with kernel $\mathcal L_G$.

Recall that for a Noetherian affine scheme $S = \Spec R$,
$\IndCoh(S)$ is a $\QCoh(S)$-module category \cite[Chapter 4, §1.2.9]{GR17I}.
If $R$ admits a dualizing complex, then this module action is just the $\IndCoh$-tensor product with $\Upsilon$.
Since $\IndCoh(S)$ is a $\QCoh(S)$-module category, one can form the relative Lurie tensor product
\[ \IndCoh(S) \tensor{\QCoh(S)} \QCoh(G).\]
\begin{lemma}
	\label{lemma: monadicity of relative tensor product}
	Assume $\pi: G \to S$ is affine.
	Then the functor 
	\[ 1 \otimes \pi_*: \IndCoh(S) \underset{\QCoh(S)}{\otimes} \QCoh(G) \to \IndCoh(S)\]
	is monadic and exhibits $\IndCoh(S)\otimes_{\QCoh(S)} \QCoh(G)$ as the category of $\pi_*\OO_G$-modules in $\IndCoh(S)$.
\end{lemma}
\begin{proof}
    Follows from \cite[Chapter 1, Corollary 8.5.7]{GR17I}
\end{proof}


Let $p: BG^\vee \to S$ and $\pi: G \to S$ be the structure maps, and let $q: S \to BG^\vee$ be the quotient map.
Let $\tilde p, \tilde \pi, \tilde q$ be defined so that the following squares are pullback squares:
\[
\begin{tikzcd}[ampersand replacement=\&]
	G \& {G \times BG^\vee} \& G \\
	S \& {BG^\vee} \& S
	\arrow["{\tilde q}", from=1-1, to=1-2]
	\arrow["\pi", from=1-1, to=2-1]
	\arrow["{\tilde p}", from=1-2, to=1-3]
	\arrow["{\tilde \pi}", from=1-2, to=2-2]
	\arrow["\pi", from=1-3, to=2-3]
	\arrow["q"', from=2-1, to=2-2]
	\arrow["p"', from=2-2, to=2-3]
\end{tikzcd}
\]

\begin{definition}
    For $G \to S$ a flat Mittag-Leffler affine commutative group scheme,
    the \emph{Fourier-Mukai transform} for $G$ is the functor 
    \begin{align*}
        \Phi_G  &: \IndCoh(S) \tensor{\QCoh(S)} \QCoh(G) \to \IndCoh(BG^\vee) \\
        \Phi_G &= \tilde \pi_* (\mathcal L_G \otimes \tilde p^!(-)).
    \end{align*}
\end{definition}
Interpreting the formula for $\Phi_G$ requires some care. The map $p$ induces a functor
\[ \tilde p^!: \IndCoh(S) \tensor{\QCoh(S)} \QCoh(G) \to \IndCoh(BG^\vee) \tensor{\QCoh(BG^\vee)} \QCoh(G \times BG^\vee)\]
given by $p^!$ in the first tensor factor and $\tilde p^*$ in the second.
After tensoring with $\mathcal{L}_G \in \QCoh(G \times BG^\vee)$, the functor $\tilde \pi_*$ then acts on the second tensor factor.

\begin{theorem}\label{theorem: FM for B of G dual}
    Let $R$ be a Noetherian ring and $G \to S =\Spec R$ be a flat Mittag-Leffler affine commutative group scheme.
    Then $\Phi_G$ is an equivalence of categories. Moreover, $\Phi_G$ fits into a commutative square 
    \[
\begin{tikzcd}[ampersand replacement=\&]
	{\IndCoh(S) \tensor{\QCoh(S)} \QCoh(G)} \& {\IndCoh(BG^\vee)} \\
	{\IndCoh(S)} \& {\IndCoh(S)}
	\arrow["{\Phi_G}", from=1-1, to=1-2]
	\arrow["{\pi_*}"', from=1-1, to=2-1]
	\arrow["{q^!}", from=1-2, to=2-2]
	\arrow["id"', from=2-1, to=2-2]
\end{tikzcd}
\]
where the horizontal functors are equivalences and the vertical functors are monadic.
\end{theorem}
\begin{proof}
    By Lemma \ref{lemma: monadicity of relative tensor product}, 
    \[\pi_*: \IndCoh(S) \underset{\QCoh(S)}{\otimes} \QCoh(G) \to \IndCoh(S)\] 
    is monadic. Since $q: S \to BG^\vee$ is ind-finite schematic, $q^!: \IndCoh(BG^\vee) \to \IndCoh(S)$ is monadic \cite[Chapter 3, Proposition 3.3.3]{GR17II}.
    
    Now we show that the square commutes.
    Since $\pi$ is affine, $\pi_*$ satisfies base change for $\QCoh$, so
    $q^! \tilde \pi_* \simeq \pi_* \tilde q^!$
    as functors 
    \[\IndCoh(BG^\vee) \tensor{\QCoh(BG^\vee)} \QCoh(G \times BG^\vee) \to \IndCoh(S).\] 
    Thus
    \[ q^! \Phi_G \simeq \pi_*(q^*\mathcal L_G \otimes \tilde q^!\tilde p^!(-)) \simeq \pi_*(q^*\mathcal L_G \otimes -).\]
    The line bundle $q^*\mathcal L_G$ has a canonical section, via our description of $\mathcal L_G$ by descent, which shows that $q^! \Phi_G \simeq \pi_*$, as desired.

    The equivalence $q^! \Phi_G \simeq \pi_*$ induces a morphism of monads 
    \begin{equation}\label{eq: the morphism of monads}
         q^!q_! \to \pi_*\pi^*
    \end{equation}
    defined by 
    \[ q^!q_! \to q^!q_!\pi_*\pi^* \simeq q^!q_!q^!\Phi_G\pi^* \to q^!\Phi_G\pi^* \simeq \pi_*\pi^*,\]
    where the first arrow is induced by the unit $1 \to \pi_*\pi^*$
    and the second arrow is induced by the counit $q_!q^! \to 1$.
    The functor $\Phi_G$ is an equivalence of categories if and only if \eqref{eq: the morphism of monads} is an equivalence \cite[Lemma 3.25]{HM21_monads}.
    
    We will show \eqref{eq: the morphism of monads} is an equivalence by computing $q^!q_!$ by base change.
    Suppose $G = \Spec A$ and let $r: G^\vee \to S$ be the structure map. Since $q$ is ind-finite ind-schematic, base change gives
    \[ q^!q_! \simeq r_!r^!.\]
    Write $G^\vee = \colim_\alpha Y_\alpha$ as a filtered colimit of finite $S$-schemes $r_\alpha: Y_\alpha \to S$.
    Then $ r_! r^! = \colim_\alpha ( r_\alpha)_!( r_\alpha^!)$.
    Since $ r_\alpha$ is a finite morphism, $( r_\alpha)^!$ is the ind-extension of the functor $\Hom(\OO_{ Y_\alpha}, -)$ on $\Coh(S)$.
    Hence, $r_!r^!$ is the ind-extension of 
    \[ \colim_\alpha \Hom(\OO_{Y_\alpha},-) = \Hom(A^\vee,-): \Coh(S) \to \IndCoh(S).\]
    Under this description, \eqref{eq: the morphism of monads} applied to some $M \in \Coh(S)$ is the composition
    \[
\begin{tikzcd}[ampersand replacement=\&]
	{\Hom_{R}(A^\vee,M)} \& {\Hom_{R}(A^\vee, A\otimes M)} \\
	\& {\Hom_{R}(A^\vee, \pi_*\tilde q^!\mathcal L \otimes M)} \& {\pi_*\tilde q^!\mathcal L \otimes M} \& {A \otimes M}
	\arrow[from=1-1, to=1-2]
	\arrow["\simeq", from=1-2, to=2-2]
	\arrow["{(\ast)}", from=2-2, to=2-3]
	\arrow["\simeq", from=2-3, to=2-4]
\end{tikzcd}
\]
where $(\ast)$ is induced by the counit $q_!q^! \to 1$.
Now for $F \in \IndCoh(BG^\vee)$, the map $q^!q_!q^!F \to q^!F$ is induced under the base change isomorphism $r_!r^!q^!F \to q^!q_!q^!F$ by the descent datum $r^!q^!F \to r^!q^!F$.
As the descent datum for $\mathcal L$ is multiplication by the canonical pairing $can \in \Gamma(G \times G^\vee, \OO)^\times$,
our map $\Hom_{R}(A^\vee,M) \to A \otimes M$ is exactly the isomorphism $\Hom_{R}(A^\vee,M) \simeq A \otimes M$ defining linear duality for flat Mittag-Leffler modules. Thus \eqref{eq: the morphism of monads} is an equivalence, proving the claim.
\end{proof}

Theorem \ref{theorem: FM for B of G dual} describes $\IndCoh(BG^\vee)$ as modules over the algebra of functions on $G$ in $\IndCoh(S)$.
When $R$ admits a dualizing complex, this description also respects the subcategories of quasi-coherent sheaves embedded via $\Upsilon$. 

\begin{corollary}
    \label{cor: qcoh of BG dual}
    Let $R$ be a Noetherian ring admitting a dualizing complex, $S = \Spec R$, and  $G \to S$ be a flat Mittag-Leffler affine commutative group scheme.
    Then the Fourier-Mukai transform for $G$ induces an equivalence 
    \[\tilde \pi_* ( \mathcal L_G \otimes \tilde p^*(-)): \QCoh(G) \overset{\sim}{\to} \QCoh(BG^\vee)\] 
    fitting into a commutative square 
\[
\begin{tikzcd}[ampersand replacement=\&]
	{\QCoh(G)} \& {\QCoh(BG^\vee)} \\
	{\QCoh(S)} \& {\QCoh(S)}
	\arrow["\sim", from=1-1, to=1-2]
	\arrow["{\pi_*}", from=1-1, to=2-1]
	\arrow["{q^*}", from=1-2, to=2-2]
	\arrow["id"', from=2-1, to=2-2]
\end{tikzcd}
\]
where the vertical maps are monadic.
\end{corollary}
\begin{proof}
    By Proposition \ref{prop: pullback detects QCoh}, $E \in \IndCoh(BG^\vee)$ is in the essential image of $\Upsilon_{BG^\vee}$ if and only if $q^!F$ is in the essential image of $\Upsilon_S$. Thus  
    \[\Phi_G: \IndCoh(S) \tensor{\QCoh(S)} \QCoh(G) \simeq \IndCoh(BG^\vee)\] 
    restricts to the desired equivalence.
\end{proof}

\begin{remark}
    Note that the equivalence of Corollary \ref{cor: qcoh of BG dual} involves composing the left adjoint functor $\tilde p^*$ with the right adjoint functor $\tilde \pi_*$. Thus, \emph{a priori}, the functor $\tilde \pi_*(\mathcal L_G \otimes \tilde p^*(-))$ is neither a left nor a right adjoint, and so there is not an obvious candidate for the inverse functor.

    This is related to our difficulty in writing down the equivalence of Theorem \ref{theorem: FM for 1-motives} as an integral transform.
    The correct functor for integrating along $G/H^\vee \to S$ should be $\pi_*(q^*)^L$ where $q: G \to G/H^\vee$ is the quotient map, $(q^*)^L$ is left adjoint to $q^*$, and $\pi: G \to S$ is the struture map. However, such a functor cannot be expressed using the usual functoriality on $\QCoh$ or $\IndCoh$, and it is not clear what kind of base change properties such functors have.
\end{remark}    

\subsection{Representations of an affine group scheme}
\label{subsection: reps of G}

In this section, let $R$ be a Noetherian ring and let $S = \Spec R$.

\begin{remark} We require $R$ to be Noetherian so that we can use the descent for ind-coherent sheaves due to D.~Gaitsgory and N.~Rozenblyum. In particular, we do not require $R$ to have a dualizing complex. It is possible that the results hold without any assumptions on the base. 
\end{remark}

If $G \to S$ is a flat Mittag-Leffler affine commutative group scheme, 
we could also ask for a Fourier-Mukai transform relating sheaves on $BG$ to sheaves on $G^\vee$.
This is outside of the paradigm of Theorem \ref{theorem: FM for 1-motives},
where the relevant categories were realized as modules over certain associative algebras.
Instead, both categories of sheaves on $BG$ and $G^\vee$ are realized as comodules over coalgebras.


For general $G$, our Fourier-Mukai transform will be defined on $\IndCoh(G^\vee)^{\geq 0}$. By definition, $\IndCoh(G^\vee)$ is equipped with a $t$-structure such that $F \in \IndCoh(G^\vee)^{\geq 0}$ if and only if $i^!F \in \IndCoh(Y)^{\geq 0}$ whenever $i: Y \to G^\vee$ is a closed embedding where $X$ is a spectral scheme of almost finite type. Given a presentation $G^\vee = \colim_\alpha Y_\alpha$ where $i_\alpha: Y_\alpha \to G^\vee$ are closed subschemes, $F \in \IndCoh(G^\vee)^{\geq 0}$ if and only if $i_\alpha^!F \in \IndCoh(Y_\alpha)^{\geq 0}$ by \cite[Chapter 3, Lemma 1.2.3]{GR17II}. In particular, it suffices to consider classical test schemes. 
Thus, $\IndCoh(G^\vee)^{\geq 0} = \lim_\alpha \IndCoh(Y_\alpha)^{\geq 0}$ along $!$-pullbacks.

Moreover, the $!$-pullback functor admits an exact left adjoint (the $\ast$-pushforward), and so in this situation
\[ \IndCoh(G^\vee)^{\geq 0} = \colim_\alpha \IndCoh(Y_\alpha)^{\geq 0}\]
along $\ast$-pushforwards \cite[Chapter 1, Proposition 2.5.7]{GR17I}. This is similar to \cite[Chapter 3, Corollary 1.1.6]{GR17II}.
Note also that $\IndCoh(Y_\alpha)^{\geq 0} \simeq \QCoh(Y_\alpha)^{\geq 0}$ since $Y_\alpha$ is a scheme \cite[Chapter 4, Proposition 1.2.2]{GR17I}.

\begin{definition}
    Let $G \to S$ be a flat Mittag-Leffler group scheme.
    The Fourier-Mukai transform for $G^\vee$ is the functor 
    \[ \Phi'_{G^\vee}: \IndCoh(G^\vee)^{\geq 0} \to \QCoh(BG)^{\geq 0}\]
    defined as follows:
    let $\mathcal L_{G^\vee}$ be the tautological line bundle on on $G^\vee \times BG$.
    As $G^\vee \to S$ is ind-finite, write $G^\vee = \colim_\alpha Y_\alpha$ where $Y_\alpha \to S$ is finite; then
    \[ 
    \Phi'_{G^\vee} := \colim_\alpha (\tilde \pi_\alpha)_*(\mathcal L_{G^\vee} \otimes \tilde p_\alpha^*(-)) : \IndCoh(G^\vee)^{\geq 0} = \colim_\alpha \QCoh(Y_\alpha)^{\geq 0} \to \QCoh(BG)^{\geq 0},
\]
where $\tilde p_\alpha$ and $\tilde \pi_\alpha$ are as below:
\[
\begin{tikzcd}[ampersand replacement=\&]
	{Y_\alpha} \& {Y_\alpha \times BG} \& {Y_\alpha} \\
	S \& BG \& S
	\arrow["{\tilde q_\alpha}", from=1-1, to=1-2]
	\arrow["{\pi_\alpha}"', from=1-1, to=2-1]
	\arrow["{\tilde p_\alpha}", from=1-2, to=1-3]
	\arrow["{\tilde \pi_\alpha}"', from=1-2, to=2-2]
	\arrow["{\pi_\alpha}"', from=1-3, to=2-3]
	\arrow["q", from=2-1, to=2-2]
	\arrow["p", from=2-2, to=2-3]
\end{tikzcd}
.
\]
\end{definition}

\begin{remark}
Let $\pi: G^\vee \to S$ be the structure map and $\tilde \pi: G^\vee \times BG \to BG$ be its base change to $BG$.
We write $\pi_! = \colim_\alpha (\pi_\alpha)_*$ and $\tilde \pi_! = \colim_\alpha (\tilde \pi_\alpha)_*$. Since the morphisms in question are (ind)-finite, there is no difference between $\ast$ and $!$. In the notation of \cite{GR17II}, these functors are $\pi_\ast^{\mathrm{IndCoh}}, \tilde \pi_\ast^{\mathrm{IndCoh}}$.

However, the definition of the functor $\Phi'_{G^\vee}$ is not written in terms of these functors in order to avoid discussing $\IndCoh(G^\vee \times BG)^{\geq 0}$, whose definition \emph{a priori} involves both limits and colimits of categories of ind-coherent sheaves. Note also that $\tilde{p}_\alpha^*$ takes $\QCoh(Y_\alpha)^{\geq 0}$ into $\QCoh(Y_\alpha \times BG)^{\geq 0}$, as can be checked by pulling back along the flat cover $Y_\alpha \to Y_\alpha \times BG$.
\end{remark}

By flat descent, $q^*: \QCoh(BG) \to \QCoh(S)$ is comonadic. 
To show that the Fourier-Mukai transform is an equivalence, our first step will be to show that $\IndCoh(G^\vee)^{\geq 0}$ is comonadic over $\QCoh(S)^{\geq 0}$.
By the Barr-Beck-Lurie theorem, a right adjoint functor is comonadic if it is conservative and preserves cosimplicial totalizations.
If the functor happens to be left t-exact with respect to some t-structure, it turns out that it must preserve bounded-below totalizations:

\begin{lemma}\label{lemma: totalizations in right-complete categories}
    Let $\mathcal{C}$ and $\mathcal{D}$ be stable $\infty$-categories, both right-complete with respect to some t-structures whose hearts are abelian.
    Let $F: \mathcal{C} \to \mathcal{D}$ be an exact and left t-exact functor.
    Let $X^\bullet$ be a cosimplicial object in $\mathcal C^{\geq 0}$.
    Then 
    \begin{enumerate}
        \item the totalization $\Tot(X^\bullet)$ exists.
        \item $F$ preserves the totalization of $X^\bullet$, that is, the natural map 
    \[ F(\Tot X^\bullet) \to \Tot F(X^\bullet)\]
    is an equivalence.
    \end{enumerate}
\end{lemma}
\begin{proof}
    Lurie's $\infty$-categorical Dold-Kan correspondence \cite[Theorem 1.2.4.1]{lurieha} gives an equivalence $X^\bullet \rightsquigarrow DK(X^\bullet)$ from the category of cosimplicial objects in $\mathcal C$ to the category of cofiltered objects in $\mathcal C$. The functor is given by
    \[ DK(X^\bullet) = \cdots \to Y(-2) \to Y(-1) \to Y(0)\]
    where $Y(-k)$ is the limit of the $k$-coskeleton of $X$.
    Further, $\Tot(X^\bullet) \simeq \lim_n Y(-n)$, in the sense that one limit exists if the other does, and they are equivalent.
    
    Let $X^\bullet$ be a cosimplicial object of $\mathcal C$ and $DK(X^\bullet) = (Y(-\bullet))$. By \cite[Proposition 1.2.4.5]{lurieha},
    the map $Y(-n) \to Y(-n+1)$ has fiber in $\mathcal C^{\geq n}$,
    so by right-completeness, the limit $\lim_n Y(-n) \simeq \Tot(X^\bullet)$ exists.
    The Dold-Kan correspondence is natural with respect to exact functors, so we have a commutative square
\[\begin{tikzcd}[ampersand replacement=\&]
	{\Tot F(X^\bullet)} \& {F\Tot(X^\bullet)} \\
	{\lim_n F(Y(-n))} \& {F(\lim_n Y(-n))}
	\arrow[from=1-1, to=1-2]
	\arrow["\simeq", from=1-1, to=2-1]
	\arrow["\simeq"', from=1-2, to=2-2]
	\arrow[from=2-1, to=2-2]
\end{tikzcd}\]
To show that $F$ preserves the totalization of $X^\bullet$, it suffices to show that the bottom arrow is an isomorphism. 
    Since $F$ is left $t$-exact and both $\mathcal C$ and $\mathcal D$ are right-complete, it suffices to prove the claim after applying the truncation $\tau^{\leq m}$.
    Thus we want to show 
    \[ \lim_n F(\tau^{\leq m}Y(-n)) \to F(\lim_n \tau^{\leq m}Y(-n))\]
    is an equivalence for all $m$.
    Since $Y(-n) \to Y(-n+1)$ has fiber in $\mathcal C^{\geq n}$, $\tau^{\leq m}Y(-n)$ is constant for $n \geq m$, so $F$ preserves the limit $\lim_n \tau^{\leq m} Y(-n)$.
\end{proof}

\begin{theorem}\label{theorem: comodules over distribution coalgebra}
    Let $R$ be a Noetherian ring and $S = \Spec R$.
    Let $\pi: X \to S$ be a coflat ind-finite ind-schematic morphism.
    Then $\pi_!: \IndCoh(X)^{\geq 0} \to \IndCoh(S)^{\geq 0}$ is comonadic.
    The resulting comonad $\pi_!\pi^!$ is identified with tensoring with the distribution coalgebra $A^\vee_X$ of $X$.
\end{theorem}
\begin{proof}
    Write $X = \colim X_\alpha$ as a colimit of schemes finite over $S$ where $i_\alpha: X_\alpha \to X$ is the inclusion and $\pi_\alpha: X_\alpha \to S$ is the structure morphism.
    We have $\IndCoh(X)^{\geq 0} = \colim_\alpha \IndCoh(X_\alpha)^{\geq 0}$ under the functors $(i_\alpha)_!$,
    and $\pi_! = \colim_\alpha (\pi_\alpha)_*$
    Each $(\pi_\alpha)_*$ is conservative on $\IndCoh^{\geq 0}$, so $\pi_!$ is conservative.
    Since $\pi$ is ind-schematic, \cite[3, Lemma 1.4.9]{GR17I}, $\pi_!: \IndCoh(X) \to \IndCoh(S)$ is left t-exact. By Lemma \ref{lemma: totalizations in right-complete categories}, $\pi_!$ preserves totalizations in $\IndCoh(X)^{\geq 0}$.
    By the Barr-Beck-Lurie theorem \cite[Theorem 4.7.3.5]{lurieha}, the functor $\pi_!: \IndCoh(X)^{\geq 0} \to \IndCoh(S)^{\geq 0}$ is comonadic.
    The functor $\pi_!\pi^!$ is given by mapping out of the pro-algebra of functions on $X$, which by linear duality is equivalent to tensoring with the distribution coalgebra $A^\vee_X$ (as defined in §\ref{subsection: indcoh}).
\end{proof}

\begin{theorem}\label{theorem: FM for BG}
    Let $R$ be a Noetherian ring and $S = \Spec R$.
    Let $G \to S$ be a flat Mittag-Leffler affine commutative group scheme.
    The Fourier-Mukai transform $\Phi'_{G^\vee}: \IndCoh(G^\vee)^{\geq 0} \to \QCoh(BG)^{\geq 0}$ is an equivalence fitting into a commutative diagram 
    \[
    \begin{tikzcd}[ampersand replacement=\&]
	{\IndCoh(G^\vee)^{\geq 0}} \& {\QCoh(BG)^{\geq 0}} \\
	{\QCoh(S)^{\geq 0}} \& {\QCoh(S)^{\geq 0}}
	\arrow["{\Phi'_{G^\vee}}"', from=1-1, to=1-2]
	\arrow["{\pi_!}", from=1-1, to=2-1]
	\arrow["{q^*}"', from=1-2, to=2-2]
	\arrow["{=}", from=2-1, to=2-2]
\end{tikzcd}
    \]
where the vertical arrows are comonadic.
\end{theorem}
\begin{proof}
    First, we construct the commutative square of functors.
    Since $\tilde \pi_\alpha: Y_\alpha \times BG \to BG$ is schematic quasicompact, we have the base change isomorphism
    \[ q^* (\tilde \pi_\alpha)_\ast \simeq (\pi_\alpha)_* (\tilde q_\alpha)^*\]
    \cite[Chapter 3, Proposition 2.2.2]{GR17I}.
    Thus, \[ q^*(\tilde \pi_\alpha)_*(\mathcal L_{G^\vee} \otimes \tilde p_\alpha^*) \simeq (\pi_\alpha)_* (\tilde q_\alpha^*\mathcal L_{G^\vee} \otimes \tilde q_\alpha^* \tilde p_\alpha^*) \simeq (\pi_\alpha)_*.\] Taking the colimit over $\alpha$ gives $q^*\Phi'_{G^\vee} \simeq \pi_!$.
    
    Since $G \to S$ is faithfully flat, $S \to BG$ is faithfully flat, so by faithfully flat descent, $q^*$ is comonadic \cite[Chapter 3, Corollary 1.3.5]{GR17I}.
    By Theorem \ref{theorem: comodules over distribution coalgebra}, $\pi_!: \IndCoh(G^\vee)^{\geq 0} \to \IndCoh(S)^{\geq 0}$ is comonadic.
    As in the proof of Theorem \ref{theorem: FM for B of G dual}, the equivalence $q^*\Phi'_{G^\vee} \simeq \pi_!$ induces a morphism 
    \[ q^*q_* \to \pi_!\pi^!,\]
    and we wish to show this morphism is an equivalence.
    If $G = \Spec A$, then 
    $q^*q_*$ is the functor $M \mapsto M \otimes A$ while 
    $\pi_!\pi^!$ is the functor $M \mapsto \Hom(A^\vee, M)$. As in Theorem \ref{theorem: FM for B of G dual}, the induced morphism $M \otimes A \to \Hom(A^\vee, M)$ is exactly the linear duality defining the pro-module $A^\vee$, and so is an equivalence. Thus $\Phi'_{G^\vee}$ is an equivalence.
\end{proof}

It is natural to ask whether the equivalence of Theorem \ref{theorem: FM for BG} can be extended from bounded below to unbounded derived categories.
Note that $\QCoh(BG)$ is left-complete \cite[1, Lemma 2.6.2]{GR17I}. However, since $\IndCoh(G^\vee)$ is rarely left-complete, Theorem \ref{theorem: FM for BG} cannot generally be extended to an equivalence of bounded derived categories.

If $\cX$ is the completion of a Noetherian affine scheme along a closed subscheme, then the following lemma describes the left completion of $\IndCoh(\cX)^+$.

\begin{lemma}
    \label{lemma: left completion of indcoh of completion}
    Let $X$ be a Noetherian affine scheme, $Y \subseteq X$ a closed subscheme.
    Let $\hat i: X^{\wedge}_Y \to X$ be the inclusion of the completion of $X$ along $Y$ into $X$. 
    
    \begin{enumerate}
        \item $\hat{i}_*: \QCoh(X_Y^\wedge) \to \QCoh(X)_Y$ is an equivalence;
        \item the functors 
    \[ 
        \IndCoh(X^\wedge_Y)^{\geq 0} \overset{\hat{i}_*}{\to} \IndCoh(X)^{\geq 0}_Y \overset{\Psi_X}{\to} \QCoh(X)^{\geq 0}_Y 
    \]
    are equivalences and exhibit $\QCoh(X_Y^\wedge) \simeq \QCoh(X)_Y$ as the left completion of $\IndCoh(X^\wedge_Y)^+$.
    \end{enumerate}
\end{lemma}
\begin{proof}
\begin{enumerate}
    \item This is \cite[Proposition 7.1.3]{GR14_indschemes}.
    \item By \cite[Proposition 7.4.5]{GR14_indschemes}, 
    \[ \hat{i}_*: \IndCoh(X^\wedge_Y) \to \IndCoh(X)_Y\]
    is an equivalence.
    By \cite[Lemma 7.4.8]{GR14_indschemes}, $\hat{i}_*$ is $t$-exact.
    Since $X$ is an affine scheme,
    the functor $\Psi_X: \IndCoh(X) \to \QCoh(X)$ defined as the ind-extension of the inclusion $\Coh(X) \to \QCoh(X)$ is an equivalence on the subcategories of coconnective objects \cite[Chapter 4, Proposition 1.2.2]{GR17I}.
    Evidently $\Psi_X$ further restricts to an equivalence between subcategories supported on $Y$.
    Since $X$ is an affine scheme, $\QCoh(X)$ is left complete, and thus $\QCoh(X)_Y$ with the induced $t$-structure is also left-complete.\qedhere
\end{enumerate}
\end{proof}
\begin{remark}
    The $t$-structure on $\QCoh(X)_Y$ used in the proof of Lemma \ref{lemma: left completion of indcoh of completion} is \emph{not} the same as the canonical $t$-structure on $\QCoh(X_Y^\wedge)$. Instead, it is what Gaitsgory and Rozenblyum call the \emph{inductive} $t$-structure. See \cite[§7.3]{GR14_indschemes}.
    It is not hard to show that these two $t$-structures have the same left-completion.
\end{remark}

\begin{theorem}
    \label{theorem: FM for dual formal Lie group}
    Let $R$ be a Noetherian ring and $S = \Spec R$.
    Let $G \to S$ be the Cartier dual of a formal Lie group.
    Then $\Phi'_{G^\vee}$ extends to an equivalence 
    \[\QCoh(G^\vee) \overset{\sim}{\to} \QCoh(BG)\]
    fitting into a commutative diagram
    \[
\begin{tikzcd}[ampersand replacement=\&]
	{\QCoh(G^\vee)} \& {\QCoh(BG)} \\
	{\QCoh(S)} \& {\QCoh(S)}
	\arrow["{\Phi'_{G^\vee}}", from=1-1, to=1-2]
	\arrow["{\pi_!}", from=1-1, to=2-1]
	\arrow["{q^*}", from=1-2, to=2-2]
	\arrow["{=}", from=2-1, to=2-2]
\end{tikzcd}
    \]
\end{theorem}
\begin{proof}
    Since $G^\vee \to S$ is a formal Lie group, 
    $G^\vee$ is of the form $X_Y^\wedge$ as an ind-scheme, where $X \to S$ is a finite rank vector bundle and $Y$ is the zero section.
    Lemma \ref{lemma: left completion of indcoh of completion} shows that $\QCoh(G^\vee)$ is the left-completion of $\IndCoh(G^\vee)^+$. 
    Thus, the equivalence of Theorem \ref{theorem: FM for BG} canonically extends to an equivalence of left completions.
\end{proof}

Theorem \ref{theorem: FM for dual formal Lie group} is a version of \cite[Theorem A]{SST25}.

\subsection{Symmetric monoidal structures}\label{subsection: symmetric monoidal structures}

We would like to know whether the Fourier-Mukai transform carries the structure of a symmetric monoidal functor exchanging convolution and the tensor product. Our categories in question are monadic, respectively comonadic, over the base category $\QCoh(S)$, respectively $\QCoh(S)^+$, as established by Theorem \ref{theorem: FM for B of G dual}, respectively \ref{theorem: FM for BG}.
So we need to describe the monoidal structure in terms of these monads or comonads.

If $\mathcal C$ is a symmetric monoidal 1-category equipped with a monad $T: \mathcal C \to \mathcal C$, then Moerdijk showed that symmetric monoidal structures on $\mathcal C^T \to \mathcal C$ are equivalent to an oplax symmetric monoidal structure on the functor $T$, that is, morphisms 
\[T(X \otimes Y) \to T(X) \otimes T(Y), \qquad T(1_{\mathcal C}) \to 1_{\mathcal C} \]
satisfying certain coherences \cite{Moe02}. If $T = A \otimes -$ where $A$ is an associative algebra in $\mathcal C$, then the data above is exactly a pair of morphisms $A \to A \otimes A$, $A \to 1_{\mathcal C}$ making $A$ into a bialgebra.
A homotopy-coherent upgrade of this statement was recently proved by Heine:

\begin{proposition}[\cite{Hei25}, Theorem 6.30]
\label{prop: symmetric monoidal structures on monads}
    Let $\Op_\infty^{\Einfty}$ be the $(\infty,2)$-category of symmetric monoidal categories, and let $\Fun([1],\Op_\infty^{\Einfty})^{\mathrm{R}}$ be the full subcategory of morphisms in $\Op_\infty^{\Einfty}$ which are right adjoint.
    There is an adjunction $$ \End: \Fun([1],\Op_\infty^{\Einfty})^{\mathrm{R}}\rightleftarrows \mathrm{Mon}(\Op_\infty^{\Einfty})^{op}: \Alg,$$ where $\Alg$ sends a monad to its Eilenberg-Moore object, and $\End$ sends a right adjoint morphism to its monad. The functor $\Alg$ is fully faithful and the essential image is the full subcategory of monadic functors.
\end{proposition}
As a consequence of Proposition \ref{prop: symmetric monoidal structures on monads}, 
we see that if $A_1$ and $A_2$ are cocommutative bialgebras in $\mathcal C$, then the space of symmetric monoidal functors $A_1\mathrm{-mod} \to A_2\mathrm{-mod}$ over $\mathcal C$ is equivalent to the space of cocommutative bialgebra morphisms $A_2 \to A_1$.

\begin{theorem}\label{theorem: FM for B of G dual is symmetric monoidal}
    Let $R$ be a Noetherian ring admitting a dualizing complex and $G \to S = \Spec R$ be a flat Mittag-Leffler affine commutative group scheme.
    The Fourier-Mukai equivalence
    \[ \Phi_G: \QCoh(G) \tensor{\QCoh(S)} \IndCoh(S) \to \IndCoh(BG^\vee)\]
    carries a symmetric monoidal structure with respect to convolution on $G$ and tensor product in $\IndCoh(BG^\vee)$.
\end{theorem}
\begin{proof}
    Let $G = \Spec A$.
    The functor 
    $\pi_*: \QCoh(G) \to \QCoh(S)$
    is symmetric monoidal with respect to convolution and the functor $q^!: \IndCoh(BG^\vee) \to \IndCoh(S)$
    is symmetric monoidal with respect to tensor product.
    By Proposition \ref{prop: symmetric monoidal structures on monads}, these symmetric monoidal structures correspond to cocommutative coalgebra structures on the monads $A \otimes -$ and $\Upsilon(A_{G^\vee}) \otimes^! -$, respectively.
    The convolution monoidal structure on $\QCoh(G)$ corresponds to the diagonal $\Delta: A \to A \otimes A$, 
    while the tensor product on $\IndCoh(BG^\vee)$ corresponds to the canonical coproduct on $\Upsilon$ linearly dual to the product on functions on $G^\vee$.
    The Fourier-Mukai transform corresponds to the isomorphism 
    \[ A \cong A_{G^\vee}\]
    of algebras in $\QCoh(S)$,
    so it suffices to endow this isomorphism with the structure of a commutative coalgebra morphism. 
    But Cartier duality exactly exchanges the coproduct on $A$ with the product on functions on $G^\vee$, as desired.
\end{proof}

\begin{remark}\label{remark: symmetric monoidal structure and BG}
In fact, Heine proved a version of Proposition \ref{prop: symmetric monoidal structures on monads} for much more general $(\infty,2)$-categories $\mathcal C$, including $(\Op_{\infty}^{\Einfty})^{op}$ \cite[Theorem 6.30]{Hei25}. Thus, if $A_1$ and $A_2$ are commutative bialgebras in $\mathcal C$, then the space of symmetric monoidal functors $A_1\mathrm{-comod} \to A_2\mathrm{-comod}$ is equivalent to the space of commutative bialgebra morphisms $A_1 \to A_2$ in $\mathcal C$.

    Unfortunately, this does not apply in the setting of Theorem \ref{theorem: FM for BG} since $\QCoh(S)^{\geq 0}$ is not a symmetric monoidal category. Presumably, this technical difficulty can be resolved by choosing an appropriate renormalization of categories. However, in the case of the Cartier dual of a formal Lie group, our Fourier-Mukai equivalence extends to an equivalence over $\QCoh(S)$, so we find that when $G$ is the Cartier dual of a formal Lie group, $\Phi': \QCoh(G^\vee) \to \QCoh(BG)$ upgrades to a symmetric monoidal functor. 
\end{remark}

%


\subsection{Naturality of the Fourier-Mukai transform}

\begin{remark}
    It is natural to expect that the Fourier-Mukai transforms constructed in this paper are natural, that is, upgrade to functors from the category of flat Mittag-Leffler affine commutative group schemes to the arrow category $(\stCats)^{\to}$ of stable cocomplete $\infty$-categories with continuous functors.
    For concreteness, let us consider the Fourier-Mukai transform 
    \[ \Phi_G = \tilde \pi_*\left(\mathcal L_G \otimes \tilde p^*-\right): \QCoh(G) \to \QCoh(BG^\vee).\]
    $\Phi_G$ is the composite of two functors $- \boxtimes \OO(1): \QCoh(G) \to \QCoh(G \times B\mathbb G_m)$ and push-pull along the correspondence 
\begin{equation}\label{eq: FM correspondence}
\begin{tikzcd}[ampersand replacement=\&]
	{G \times BG^\vee} \& {G \times B\Gm} \\
	{BG^\vee}
	\arrow[from=1-1, to=1-2]
	\arrow[from=1-1, to=2-1]
\end{tikzcd}
\end{equation}
where the map $G \times BG^\vee \to G$ is projection onto the first factor, while $G \times BG^\vee \to B\Gm$ is the pairing corresponding to $\mathcal L_G$.
Gaitsgory and Rozenblyum showed that the assignment $\cX \rightsquigarrow \QCoh(\cX)$ upgrades to a functor 
\[ \mathrm{Corr}_{sch, all} \to \mathrm{Cat_\infty}^{\mathrm{st,cocompl}}_{\mathrm{cont}}\]
where $\mathrm{Corr}_{sch,all}$ is the category of correspondences of prestacks, where vertical arrows are schematic \cite[Chapter 5, §5.3]{GR17I}. 
Thus, to show that $\Phi_G$ upgrades to the desired functor, it suffices to show that the assignment of $G$ to the correspondence \eqref{eq: FM correspondence} upgrades to a functor from flat Mittag-Leffler affine commutative group schemes to $\mathrm{Corr}^\to$, the arrow category of correspondences of prestacks.

Such an upgrade would automatically provide an upgrade of $\Phi_G$ to a symmetric monoidal transformation exchanging convolution with tensor product as in §\ref{subsection: symmetric monoidal structures}.
\end{remark}

\appendix 

\section{Ind-finite descent for perfect complexes}

\subsection{h-topology}\label{subs: h-top}

\begin{definition}[\cite{Cho25}, Definition 2.1-2.2]\label{defn: v-cover}
    A morphism $f: A \to B$ of discrete rings is a \emph{$v$-cover} if whenever $V$ is a valuation ring and $A \to V$ is a morphism of rings, there exists an extension of valuation rings $V \subseteq W$ and a morphism $B \to W$ such that the obvious diagram is commutative: 
    \[
\begin{tikzcd}[ampersand replacement=\&]
	A \& B \\
	V \& W
	\arrow[from=1-1, to=1-2]
	\arrow[from=1-1, to=2-1]
	\arrow[dashed, from=1-2, to=2-2]
	\arrow["\subseteq", dashed, from=2-1, to=2-2]
\end{tikzcd}
.\]
    The morphism $f$ is an \emph{$h$-cover} if it is of finite presentation and a $v$-cover.

    A morphism $f: A \to B$ of connective $\Einfty$-rings is a $v$-cover (resp. $h$-cover) if $\pi_0(f): \pi_0(A) \to \pi_0(B)$ is a $v$-cover (resp. $h$-cover).
\end{definition}

By \cite[Lemma 2.5]{Cho25}, $h$-covers of rings are closed under composition and pushouts.
By \cite[Remark 2.5]{Ryd10}, fppf covers are $h$-covers, and finite morphisms surjective on geometric points are $h$-covers.

\begin{definition}\label{def: surjective in h-topology}
A morphism of prestacks $f: \cX \to \mathcal Y$ is \emph{surjective in $h$-topology} if whenever $y: \Spec A \to \mathcal Y$, there exists an $h$-cover $\phi: \Spec B \to \Spec A$ such that $\phi^*y$ is in the essential image of $f$.
\end{definition}

Since surjective finite morphisms are $h$-covers, it is natural to ask whether surjective ind-finite morphisms are $h$-covers.
\begin{lemma}
    \label{lemma: ind-finite noetherian surjective on geometric pts}
    Suppose that $\mathcal Y$ is a locally Noetherian prestack and that $f: \cX \to \mathcal Y$ is an ind-finite ind-schematic morphism of prestacks which is surjective on geometric points.
    Then $f$ is surjective in $h$-topology.
\end{lemma}
\begin{proof}
    Let $A$ be a Noetherian ring and $\Spec A \to \mathcal Y$ be an $A$-point.
    Let $\cX' = \cX \times_{\mathcal Y} \Spec A$.    
    Then $\cX' \to \Spec A$ is an ind-finite ind-scheme surjective on geometric points.
    Let $\mathfrak{p}_1,\ldots, \mathfrak{p}_n$ be the minimal primes of $A$.
    Then $\coprod_i \Spec(A / \mathfrak p_i) \to \Spec A$ is a finite cover surjective on geometric points, hence an $h$-cover.
    Thus we may assume that $\Spec A$ is integral.
    Let $\eta$ be the generic point of $A$.
    Since $\cX' \to \Spec A$ is surjective on geometric points, 
    there exists a closed subscheme $\Spec B \to \cX'$ finite over $\Spec A$ such that $\eta$ is in the image of $\Spec B$.
    By the going-up theorem, $\Spec B \to \Spec A$ is surjective on geometric points, and thus is an $h$-cover, as desired.    
\end{proof}

\subsection{Descent}

Our main concern with the $h$-topology is descent for perfect complexes.
In the presence of a dualizing complex, this follows from descent for $\IndCoh$.
Halpern-Leistner and Preygel established a descent theorem for $h$-covers of locally Noetherian derived stacks \cite{HP23}. 
Chough established the following version for general affine morphisms:

\begin{theorem}[\cite{Cho25}, Theorem 1.11]
\label{theorem: h-descent for affine morphisms}
    Let $A \to B$ be an $h$-cover of connective $\Einfty$-rings.
    Then for all morphisms $A \to A'$, if $B' = B \Lotimes{A} A'$, the induced functor
    \[ \Perf(A') \to \Tot \Perf((B')^{\bullet})\]
    is an equivalence of $\infty$-categories, where $(B')^\bullet$ is the \v{C}ech nerve of $A' \to B'$.
\end{theorem}

The main result of this section is that Theorem \ref{theorem: h-descent for affine morphisms} implies a similar statement for morphisms of prestacks surjectve in $h$-topology.

\begin{theorem}\label{theorem: h-descent}
    Suppose that $\cX \to \mathcal Y$ is a morphism of prestacks surjective in $h$-topology.
    Then there is an equivalence of categories 
    \[ 
        \Perf(\mathcal Y) \to \Tot \Perf(\cX^{\bullet})
    \]
    where $\cX^\bullet$ is the \v{C}ech nerve of $\cX \to \mathcal Y$.
\end{theorem}
    The statement is likely well-known to experts, but we could not find a reference.
    It would be natural to prove it in the following way: sheafify the prestacks $\cX$ and $\mathcal Y$ in the h-topology, and then show that the map \[|(\cX^{h})^\bullet|\to\mathcal Y^h\] is an equivalence, where the upper index $h$ stands for sheafification in $h$-topology, and $|(\mathcal X^{h})^\bullet|$ is the geometric realization of the \v{C}ech nerve of $\mathcal X \to \mathcal Y$.
    An issue with this approach is that $\Perf$ does not satisfy descent with respect to h-hypercovers \cite[Remark 1.14]{Cho25}, so it is not true that the map $\Perf(\mathcal X^{h})\to\Perf(\mathcal X)$ is an equivalence. 
    Therefore, instead of using the general framework, we prove the claim directly following the same logic.
\begin{proof}[Proof of Theorem \ref{theorem: h-descent}]
    We will say that a morphism of prestacks $f:\mathcal X\to\mathcal Y$ satisfies \emph{descent for $\Perf$} if the induced functor 
     \[ 
        \Perf(\mathcal Y) \to \Tot \Perf(\mathcal X^{\bullet})
    \]
    is an equivalence. We say that $f$ satisfies \emph{universal descent for $\Perf$} if any base change of $f$ satisfies descent for $\Perf$. It is easy to see that it suffices to consider base changes under maps $Y\to\mathcal Y$ where $Y$ is an affine scheme.

    We now have the following lemmas:
    \begin{lemma}\label{lemma:Gr-section} If $f:\mathcal X\to\mathcal Y$ admits a section, it satisfies descent for $\Perf$.
    \end{lemma}
    \begin{proof}
    If $f$ admits a section, then the \v{C}ech nerve of $f$ forms a split simplicial object \cite[\href{https://kerodon.net/tag/05HN}{Tag 05HN}]{kerodon}, and split simplicial objects are universal simplicial colimits \cite[\href{https://kerodon.net/tag/04TL}{Tag 04TL}]{kerodon}. 
    \end{proof}

    \begin{lemma}\label{lemma:Gr-refinement} Consider two maps $f_1:\mathcal X_1\to\mathcal Y$ and $f_2:\mathcal X_2\to\mathcal Y$,
    and put $\mathcal X:=\mathcal X_1\times_{\mathcal Y}\mathcal X_2$. Suppose that the base change
    \[(\mathcal X_1)_{\mathcal Y}^k\times_{\mathcal Y}\mathcal X_2\to(\mathcal X_1)^k_{\mathcal Y}\]
    satisfies descent for $\Perf$ for every $k>0$. Then $f_1$ satisfies descent if and only if the fiber product $\mathcal X\to\mathcal Y$ satisfies descent.
    \end{lemma}
    \begin{proof}
Contemplate the following diagram:
\[
\begin{tikzcd}[ampersand replacement=\&]
	{\Perf(\mathcal{Y})} \& {\lim\limits_{\Delta} \Perf(\mathcal X_1^\bullet)} \\
	{\lim\limits_{\Delta} \Perf ((\mathcal X_1 \times \mathcal X_2)^\bullet)} \& {\lim\limits_{\Delta \times \Delta}(\Perf(\mathcal X_1^\bullet \times \mathcal X_2^\bullet))}
	\arrow[from=1-1, to=1-2]
	\arrow[from=1-1, to=2-1]
	\arrow[from=1-2, to=2-2]
	\arrow[from=2-1, to=2-2]
\end{tikzcd}
.
\]
By \cite[Lemma 5.5.8.4]{Lur09} the bottom arrow is an equivalence.
The hypothesis implies the right arrow is an equivalence. 
Thus one of the remaining two arrows is an equivalence if and only if the other is.
\end{proof}

    The statement now follows. Indeed, we may assume that $\mathcal Y=Y$ is an affine scheme. By the hypothesis, there is an affine scheme $Y'$ and a morphism $Y'\to\cX$ such that the composition $Y'\to\cX\to Y$ is an h-cover $Y'\to Y$. 
    For every $k>0$, the map 
    \[(Y')^k_Y\times_Y\mathcal X\to (Y')^k_Y\]
    admits a section, and therefore satisfies descent for $\Perf$ by Lemma~\ref{lemma:Gr-section}. Applying Lemma~\ref{lemma:Gr-refinement}, we see that $Y'\times_Y\mathcal X\to Y$ satisfies descent for $\Perf$ as well. Now, using Lemma~\ref{lemma:Gr-refinement} again, we see that $\mathcal X\to Y$ satisfies descent for $\Perf$, as claimed.
\end{proof}

\begin{corollary}
    \label{corollary: h-descent for ind-finite}
    Suppose that $Y$ is a Noetherian scheme and $X \to Y$ is an ind-finite ind-schematic morphism which is surjective on geometric points. 
    Then $X \to Y$ satisfies descent for $\Perf$.
\end{corollary}
\begin{proof}
    Combine Lemma \ref{lemma: ind-finite noetherian surjective on geometric pts} and Theorem \ref{theorem: h-descent}.
\end{proof}

\printbibliography




\end{document}